\documentclass[11pt]{article}%
\usepackage{amssymb}
\usepackage{graphicx}
\usepackage{amsmath}
\usepackage{amsfonts}%
\providecommand{\U}[1]{\protect\rule{.1in}{.1in}}
\newtheorem{theorem}{Theorem}[section]

\newtheorem{lemma}[theorem]{Lemma}

\newtheorem{proposition}[theorem]{Proposition}
\newtheorem{remark}[theorem]{Remark}

\newenvironment{proof}[1][Proof]{\textbf{#1.} }{\hfill\rule{0.5em}{0.5em}}
{\catcode`\@=11\global\let\AddToReset=\@addtoreset
\AddToReset{equation}{section}

\AddToReset{theorem}{section}

\begin{document}

\title{The p-Laplace heat equation with a source term : self-similar solutions revisited}
\date{.}
\author{Marie Fran\c{c}oise Bidaut-V\'{e}ron\thanks{Laboratoire de Math\'{e}matiques
et Physique Th\'{e}orique, CNRS UMR 6083, Facult\'{e} des Sciences, Parc
Grandmont, 37200 Tours, France. e-mail:veronmf@univ-tours.fr}}
\maketitle

\begin{abstract}
We study the self-similar solutions of any sign of the equation
\[
u_{t}-div(\left\vert \nabla u\right\vert ^{p-2}\nabla u)=\left\vert
u\right\vert ^{q-1}u,
\]
in $\mathbb{R}^{N},$ where $p,q>1.$ We extend the results of Haraux-Weissler
obtained for $p=2$ to the case $q>p-1>0.$ In particular we study the existence
of slow or fast decaying solutions. For given $t>0$, the fast solutions
$u(t,.)$ have a compact support in $\mathbb{R}^{N}$ when $p>2,$ and
$\left\vert x\right\vert ^{p/(2-p)}u(t,x)$ is bounded at infinity when $p<2.$
We describe the behaviour for large $\left\vert x\right\vert $ of all the
solutions. According to the position of $q$ with respect to the first critical
exponent $p-1+p/N$ and the critical Sobolev exponent $q^{\ast}$, we study the
existence of positive solutions, or the number of the zeros of $u(t,.)$. We
prove that any solution $u(t,.)$ is oscillatory when $p<2$ and $q$ is closed
to $1.$

\end{abstract}
\tableofcontents

.\pagebreak

\section{ Introduction and main results\label{intro}}

In this paper we study the existence of self-similar solutions of degenerate
parabolic equations with a source term, involving the $p$-Laplace operator in
$\mathbb{R}^{N}\times\left(  0,\infty\right)  ,$ $N\geq1,$
\begin{equation}
u_{t}-div(\left\vert \nabla u\right\vert ^{p-2}\nabla u)=\left\vert
u\right\vert ^{q-1}u, \label{lap}%
\end{equation}
where $p>1,q>1.$ The semilinear problem, relative to the case $p=2,$\textit{ }%
\begin{equation}
u_{t}-\Delta u=\left\vert u\right\vert ^{q-1}u, \label{ord}%
\end{equation}
has been treated by \cite{HW}, and \cite{W1}, \cite{W2}, \cite{PTW}. In
particular, for any $a>0,$ there exists a self-similar solution of the form
\[
u=t^{-1/(q-1)}\omega(t^{-1/2}\left\vert x\right\vert )
\]
of (\ref{ord}), unique, such that $\omega\in C^{2}(\left[  0,\infty\right)
),$ $\omega(0)=a$ and $\omega^{\prime}(0)=0.$ Any solution of this form
satisfies $\lim_{\left\vert \xi\right\vert \rightarrow\infty}\left\vert
\xi\right\vert ^{2/(q-1)}\omega(\xi)=L\in\mathbb{R}.$ It is called
\textit{slowly decaying} if $L\neq0$ and \textit{fast decaying }if $L=0.$ Let
us recall the main results:$\medskip$

$\bullet$ \textit{If }$(N+2)/N<q,$\textit{ there exist positive solutions}%
$.\medskip$

$\bullet$\textit{ If }$(N+2)/N<q<(N+2)/(N-2),$\textit{ there exist positive
solutions of each type; in particular there exists a fast decaying one with an
exponential decay: }%
\[
\lim_{\left\vert z\right\vert \rightarrow\infty}e^{\left\vert z\right\vert
^{2}/4}\left\vert z\right\vert ^{N-2/(q-1)}\omega(z)=A\in\mathbb{R},
\]
\textit{thus the solution }$u$\textit{ of (\ref{ord}) satisfies }$u(.,t)\in
L^{s}(R^{N})$\textit{ for any }$s\geq1,$ \textit{and }$\lim_{t\rightarrow
0}\left\Vert u(.,t)\right\Vert _{s}=0$\textit{ whenever }$s<N(q-1)/2,$ and
\textit{ }$\lim_{t\rightarrow0}\sup_{\left\vert x\right\vert \geq\varepsilon
}\left\vert u(x,t)\right\vert =0$\textit{ for any }$\varepsilon>0.$\textit{
Moreover for any integer }$m\geq1,$\textit{ there exists a fast decaying
solution }$\omega$ \textit{with precisely }$m$\textit{ zeros.}$\medskip$

$\bullet$\textit{ If }$(N+2)/(N-2)\leq q$\textit{ all the solutions }%
$\omega\not \equiv 0$ \textit{have a constant sign and a slow decay.}%
$\medskip$

$\bullet$\textit{ If }$q\leq(N+2)/N,$\textit{ then all the solutions }%
$\omega\not \equiv 0$ \textit{have a finite positive number of zeros, and
there exists an infinity of solutions of each type.}$\medskip$

The uniqueness of the positive fast decaying solution was proved later in
\cite{Ya} and \cite{DoHi}, and more results about the solutions can be found
in \cite{HiYa1}, \cite{Hi} and \cite{HiYa2}.\smallskip

Next we assume $p\neq2.$ If $u$ is a solution of (\ref{lap}), then for any
$\alpha_{0},\beta_{0}\in\mathbb{R},$ $u_{\lambda}(x,t)=\lambda^{\alpha_{0}%
}u(\lambda x,\lambda^{\beta_{0}}t)$ is a solution if and only if
\[
\alpha_{0}=p/(q+1-p),\qquad\beta_{0}=(q-1)\alpha_{0},
\]
This leads to search self-similar solutions of the form%
\begin{equation}
u(x,t)=(\beta_{0}t)^{-1/(q-1)}w(r),\qquad r=(\beta_{0}t)^{-1/\beta_{0}%
}\left\vert x\right\vert , \label{for}%
\end{equation}
the equation reduces to
\begin{equation}
\left(  \left\vert w^{\prime}\right\vert ^{p-2}w^{\prime}\right)  ^{\prime
}+\frac{N-1}{r}\left\vert w^{\prime}\right\vert ^{p-2}w^{\prime}+rw^{\prime
}+\alpha_{0}w+\left\vert w\right\vert ^{q-1}w=0\qquad\text{in }\left(
0,\infty\right)  . \label{pre}%
\end{equation}
In the sequel, some critical exponents are involved:%
\[
p_{1}=\frac{2N}{N+1},\qquad p_{2}=\frac{2N}{N+2},
\]%
\[
q_{1}=p-1+\frac{p}{N},\qquad q^{\ast}=\frac{N(p-1)+p}{N-p};
\]
with the convention $q^{\ast}=\infty$ if $N\leq p.$ Observe that
$p-1<q_{1}<q^{\ast};$ moreover $p_{1}<p\Leftrightarrow1<q_{1},$ and
$p_{2}<p\Leftrightarrow1<q^{\ast}$.\medskip\ We also set
\begin{equation}
\delta=\frac{p}{2-p}.\quad\text{and \quad}\eta=\frac{N-p}{p-1}, \label{del}%
\end{equation}
thus $\delta>0\Longleftrightarrow$ $p<2.$ Notice that
\begin{align}
p_{1}  &  <p<2\Longleftrightarrow N<\delta\Longleftrightarrow\eta
<N,\label{dn}\\
p_{2}  &  <p<2\Longleftrightarrow N<2\delta. \label{ddn}%
\end{align}

Problem (\ref{lap}) was studied before in \cite{Qi}. In the range
$q_{1}<q<q^{\star}$ and $p_{1}<p,$ the existence of a nonnegative solution $u$
was claimed, such that $w$ has a compact support when $p>2,$ or $w>0$ when
$p<2,$ with $w(z)=o($ $\left\vert z\right\vert ^{(-p+\varepsilon)/(2-p)})$ at
infinity, for any small $\varepsilon>0.$ However some parts of the proofs are
not clear. The equation was studied independently for $p>2$ in \cite{BG}, but
the existence of a nonnegative solution with compact support was not
established, and some proofs are incomplete. Here we clarify and improve the
former assertions, treat the case $p\leq p_{1},$ and give new informations on
the existence of changing sign solutions. In particular a new phenomenon
appears, namely the possible existence of an infinity of zeros of $w.$ Also
all the solutions have a constant sign when $p\leq p_{2}.$

\begin{theorem}
\label{prin} Let $q>\max(1,p-1).$ (i) For any $a>0,$ there exists a
self-similar solution of the form
\begin{equation}
u(t,x)=(\beta_{0}t)^{-1/(q-1)}w((\beta_{0}t)^{-1/\beta_{0}}\left\vert
x\right\vert ) \label{ssi}%
\end{equation}
of (\ref{pre}), unique, such that $w\in C^{2}\left(  \left(  0,\infty\right)
\right)  \cap C^{1}\left(  \left[  0,\infty\right)  \right)  ,$ $w(0)=a$ and
$w^{\prime}(0)=0.$ Any solution of this form satisfies $\lim_{\left\vert
z\right\vert \rightarrow\infty}\left\vert z\right\vert ^{\alpha_{0}}%
w(z)=L\in\mathbb{R}.\medskip$

\noindent(ii) If $q_{1}<q,$ there exists positive solutions with $L>0,$ also
called slow decaying.$\medskip$

\noindent(iii) If $q_{1}<q<q^{\star},$ there exists a nonnegative solution
$w\not \equiv 0$ such that $L=0,$ called fast decaying, and
\[
u(t)\in L^{s}(\mathbb{R}^{N})\quad\text{for any }s\geq1,\qquad\lim
_{t\rightarrow0}\left\Vert u(t)\right\Vert _{s}=0\quad\text{whenever
}s<N/\alpha_{0},
\]%
\[
\lim_{t\rightarrow0}\sup_{\left\vert x\right\vert \geq\varepsilon}\left\vert
u(x,t)\right\vert =0\quad\text{for any }\varepsilon>0.
\]

\noindent More precisely, when $p>2,$ $w$ has a compact support in $\left(
0,\infty\right)  ;$ when $p<2,$ $w$ is positive and
\begin{equation}%
\begin{array}
[c]{c}%
\lim_{\left\vert z\right\vert \rightarrow\infty}\left\vert z\right\vert
^{p/(2-p)}w(z)=\ell(N,p,q)>0\qquad\qquad\text{if\quad}p_{1}<p<2,\\
\lim_{\left\vert z\right\vert \rightarrow\infty}\left\vert z\right\vert
^{(N-p)/(p-1)}w(z)=c>0\qquad\qquad\qquad\text{if\quad}1<p<p_{1},\\
\lim_{r\rightarrow\infty}r^{N}(\ln r)^{(N+1)/2}w=\varrho(N,p,q)>0\qquad
\qquad\text{if\quad}p=p_{1}\text{.}%
\end{array}
\label{kil}%
\end{equation}

\noindent(iv) If $q_{1}<q<q^{\star},$ for any integer $m\geq1,$ there exists a
fast decaying solution $w\not \equiv 0$ with at least $m$ isolated zeros and a
compact support when $p>2$; there exists a fast decaying solution $w$
precisely $m$ zeros, and $\left\vert w\right\vert $ has the behaviour
(\ref{kil}) when $p<2$.$\medskip$

\noindent(v) If $p\leq p_{2},$ or if $p>p_{2}$ and $q\geq q^{\star},$ all the
solutions $w\not \equiv 0$ have a constant sign and are slowly
decaying.$\medskip$

\noindent(vi) If $q\leq q_{1},$ (hence $p_{1}<p),$ all the solutions
$w\not \equiv 0$ assume both positive and negative values. There exists an
infinity of fast decaying solutions, such that $w$ has a compact support when
$p>2,$ and $\left\vert z\right\vert ^{p/(2-p)}w(z)$ is bounded near $\infty$
when $p<2.$ Moreover if $p<2,$ and $q$ is close to $q_{1},$ and $p$ close to
$2,$ then all the solutions $w\not \equiv 0$ have a finite number of zeros. If
$p<2$ and $q$ is close to $1,$ all of them are oscillatory.
\end{theorem}

In the sequel we study more generally the equation%
\begin{equation}
\left(  \left\vert w^{\prime}\right\vert ^{p-2}w^{\prime}\right)  ^{\prime
}+\frac{N-1}{r}\left\vert w^{\prime}\right\vert ^{p-2}w^{\prime}+rw^{\prime
}+\alpha w+\left\vert w\right\vert ^{q-1}w=0\qquad\text{in }\left(
0,\infty\right)  , \label{un}%
\end{equation}
where $\alpha>0$ is a parameter, and we only assume $q>1$. The problem without
source
\begin{equation}
u_{t}-div(\left\vert \nabla u\right\vert ^{p-2}\nabla u)=0 \label{san}%
\end{equation}
was treated in \cite{QW} when $p<2$ for positive solutions. In \cite{Bi1} we
make a complete description of the solutions of any sign of (\ref{san}) for
$p<2$, and study the equation
\begin{equation}
\left(  \left\vert w^{\prime}\right\vert ^{p-2}w^{\prime}\right)  ^{\prime
}+\frac{N-1}{r}\left\vert w^{\prime}\right\vert ^{p-2}w^{\prime}+rw^{\prime
}+\alpha w=0\qquad\text{in }\left(  0,\infty\right)  , \label{hog}%
\end{equation}
for arbitrary $\alpha\in\mathbb{R}$. A main point is that equation (\ref{un})
appears as a perturbation of (\ref{hog}) when $w$ is small enough. When
$\alpha>0$ and $(\delta-N)(\delta-\alpha)>0,$ observe that (\ref{hog}) has a
particular solution of the form $w(r)=\ell r^{-\delta},$ where
\begin{equation}
\ell=\left(  \delta^{p-1}\frac{\delta-N}{\delta-\alpha}\right)  ^{1/(2-p)}.
\label{ell}%
\end{equation}
A critical value of $\alpha$ appears in studying (\ref{hog}) when $p_{2}<p:$%
\begin{equation}
\alpha^{\ast}=\delta+\frac{\delta(N-\delta)}{(p-1)(2\delta-N)}, \label{eto}%
\end{equation}
In the case $p>2,$ eqution (\ref{hog}) is treated in \cite{GV} and
\cite{Bi2}.$\medskip$

Our paper is organized as follows:$\medskip$

In Section \ref{S2}, we give general properties about equation (\ref{un}).
Among the solutions defined on $\left(  0,\infty\right)  ,$ we show the
existence and uniqueness of global solutions $w=w(.,a)\in C^{2}\left(  \left(
0,\infty\right)  \right)  \cap C^{1}\left(  \left[  0,\infty\right)  \right)
$ of problem (\ref{un}) such that for some $a\in\mathbb{R}$
\begin{equation}
w(0)=a,\qquad w^{\prime}(0)=0. \label{ini}%
\end{equation}
By symmetry, we restrict to the case $a\geq0.$ We give the first informations
on the number of zeros of the solutions, and upper estimates near $\infty$ of
any solution of any sign.\medskip

In Section \ref{S3}, we study the case $(2-p)\alpha<p.$ We first show that any
solution $w$ satisfies $\lim_{r\rightarrow\infty}r^{\alpha}w=L\in\mathbb{R}.$
Moreover we prove that the function $a\longmapsto L(a)=\lim_{r\rightarrow
\infty}r^{\alpha}w(r,a)$ is continuous on $\mathbb{R}.$ When $L=0,$ then any
solution $w$ has a compact support if $p>2,$ and $r^{\delta}w$ is bounded if
$p<2$ and we give a complete description of behaviour of $w$ near infinity.
Then we study the existence of fast decaying solutions of equation \ref{un},
positive or changing sign, according to the value of $\alpha,$ see theorems
\ref{fast} and \ref{pn}. We give sufficient conditions on $p,q,\alpha,$ in
order that all the functions $w(.,a)$ are positive and slowly decaying, see
Theorem \ref{pos}; some of them are new, even in the case $p=2.$ Finally we
prove that all the solutions $w$ are oscillatory when $p_{1}<p<2$ and $\alpha$
is close to $\delta$, see Theorem \ref{osci}; this type of behaviour never
occurs in the case $p=2.$\medskip

In Section \ref{S4} we study the case $p\leq(2-p)\alpha,$ for which equation
(\ref{un}) has no more link with problem (\ref{lap}), but is interesting in
itself. Here $r^{\delta}w$ is bounded at $\infty$, except in the case
$p=(2-p)\alpha<p_{1}$ where a logarithm appears. Moreover if $p_{1}<p,$ or
$p_{1}=p<(2-p)\alpha,$ then all the solutions are oscillatory. As in section 3
we study the existence of positive solutions, see Theorems \ref{fast2} and
\ref{pol}. At Theorem \ref{hard} we prove a difficult result of convergence in
the range $\alpha<\eta$ where the solutions are nonoscillatory.\medskip

Section \ref{S5} is devoted to the proof of Theorem \ref{prin}, by taking
$\alpha=\alpha_{0}$ and applying the results of Section \ref{S3}, since
$(2-p)\alpha_{0}<p.$\medskip

\section{General properties\label{S2}}

\subsection{Equivalent formulations, and energy functions\label{S21}}

Equation (\ref{un}) can be written under equivalent forms,
\begin{equation}
\left(  r^{N-1}\left\vert w^{\prime}\right\vert ^{p-2}w^{\prime}\right)
^{\prime}+r^{N-1}(rw^{\prime}+\alpha w+\left\vert w\right\vert ^{q-1}%
w)=0\qquad\text{in }\left(  0,\infty\right)  , \label{var}%
\end{equation}

\begin{equation}
\left(  r^{N}(w+r^{-1}\left\vert w^{\prime}\right\vert ^{p-2}w^{\prime
})\right)  ^{\prime}+r^{N-1}\left(  (\alpha-N)w+\left\vert w\right\vert
^{q-1}w\right)  =0\qquad\text{in }\left(  0,\infty\right)  . \label{vir}%
\end{equation}
Defining
\begin{equation}
J_{N}(r)=r^{N}\left(  w+r^{-1}\left\vert w^{\prime}\right\vert ^{p-2}%
w^{\prime}\right)  , \label{gg}%
\end{equation}
then (\ref{vir}) is equivalent to
\begin{equation}
J_{N}^{\prime}(r)=r^{N-1}(N-\alpha-\left\vert w\right\vert ^{q-1})w.
\label{jpn}%
\end{equation}
We also use the function
\begin{equation}
J_{\alpha}(r)=r^{\alpha}\left(  w+r^{-1}\left\vert w^{\prime}\right\vert
^{p-2}w^{\prime}\right)  =r^{\alpha-N}J_{N}(r), \label{jk}%
\end{equation}
which satisfies
\begin{equation}
J_{\alpha}^{\prime}(r)=r^{\alpha-1}\left(  (\alpha-N)r^{-1}\left\vert
w^{\prime}\right\vert ^{p-2}w^{\prime}-\left\vert w\right\vert ^{q-1}w\right)
. \label{jpk}%
\end{equation}
The simplest energy function,
\begin{equation}
E(r)=\frac{1}{p^{\prime}}\left\vert w^{\prime}\right\vert ^{p}+\frac{\alpha
}{2}w^{2}+\frac{\left\vert w\right\vert ^{q+1}}{q+1}, \label{heu}%
\end{equation}
obtained by multiplying (\ref{un}) by $w^{\prime},$ is nonincreasing, since
\begin{equation}
E^{\prime}(r)=-(N-1)r^{-1}\left\vert w^{\prime}\right\vert ^{p}-rw^{\prime2},
\label{hep}%
\end{equation}
More generally we introduce a Pohozaev-Pucci-Serrin type function with
parameters $\lambda>0,\sigma,e\in\mathbb{R}:$%
\begin{equation}
V_{\lambda,\sigma,e}(r)=r^{\lambda}\left(  \frac{\left\vert w^{\prime
}\right\vert ^{p}}{p^{\prime}}+\frac{\left\vert w\right\vert ^{q+1}}%
{q+1}+e\frac{w^{2}}{2}+\sigma r^{-1}w\left\vert w^{\prime}\right\vert
^{p-2}w^{\prime}\right)  . \label{vla}%
\end{equation}
Such kind of functions have been used intensively in \cite{PS}. After
computation we find%
\begin{align}
r^{1-\lambda}V_{\lambda,\sigma,e}^{\prime}(r)  &  =-(N-1-\sigma-\frac{\lambda
}{p^{\prime}})\left\vert w^{\prime}\right\vert ^{p}-\left(  \sigma
-\frac{\lambda}{q+1}\right)  \left\vert w\right\vert ^{q+1}+\sigma
(\lambda-N)r^{-1}w\left\vert w^{\prime}\right\vert ^{p-2}w^{\prime}\nonumber\\
&  -\left(  rw^{\prime}+\frac{\sigma-e+\alpha}{2}w\right)  ^{2}-(\sigma
\alpha-\frac{e\lambda}{2}-\frac{(\sigma+\alpha-e)^{2}}{4})w^{2}. \label{dif}%
\end{align}
Notice that $E=V_{0,0,\alpha}.\medskip$

In all the sequel we use a logarithmic substitution; for given $d\in
\mathbb{R},$
\begin{equation}
w(r)=r^{-d}y_{d}(\tau),\qquad\tau=\ln r. \label{cge}%
\end{equation}
We get the equation, at each point $\tau$ such that $w^{\prime}(r)\neq0,$%
\[
y_{d}^{\prime\prime}+(\eta-2d)y_{d}^{\prime}-d(\eta-d)y_{d}\qquad\qquad
\qquad\qquad\qquad\qquad\qquad\qquad\qquad\qquad\qquad\qquad\qquad\qquad\qquad
\]%
\begin{equation}
+\frac{1}{p-1}e^{((p-2)d+p)\tau}\left\vert dy_{d}-y_{d}^{\prime}\right\vert
^{2-p}\left(  y_{d}^{\prime}-(d-\alpha)y_{d}+e^{-d(q-1)\tau}\left\vert
y_{d}\right\vert ^{q-1}y_{d}\right)  =0. \label{phis}%
\end{equation}
Setting
\begin{equation}
Y_{d}(\tau)=-r^{(d+1)(p-1)}\left\vert w^{\prime}\right\vert ^{p-2}w^{\prime},
\label{cgc}%
\end{equation}
we can write (\ref{phis}) as a system:%
\begin{equation}
\left\{
\begin{array}
[c]{c}%
y_{d}^{\prime}=dy_{d}-\left\vert Y_{d}\right\vert ^{(2-p)/(p-1)}Y_{d}%
,\qquad\qquad\qquad\qquad\qquad\qquad\qquad\qquad\qquad\qquad\qquad\\
Y_{d}^{\prime}=(p-1)(d-\eta)Y_{d}+e^{(p+(p-2)d)\tau}(\alpha y_{d}%
+e^{-\delta(q-1)\tau}\left\vert y_{d}\right\vert ^{q-1}y_{d}-\left\vert
Y_{d}\right\vert ^{(2-p)/(p-1)}Y_{d}),
\end{array}
\right.  \label{sysd}%
\end{equation}
In particular the case $d=\delta$ plays a great role: setting
\begin{equation}
w(r)=r^{-\delta}y(\tau),\quad Y(\tau)=-r^{(\delta+1)(p-1)}\left\vert
w^{\prime}\right\vert ^{p-2}w^{\prime},\qquad\tau=\ln r, \label{cha}%
\end{equation}
equation (\ref{phis}) takes the form%
\begin{equation}
(p-1)y^{\prime\prime}+(N-\delta p)y^{\prime}+(\delta-N)\delta y+\left\vert
\delta y-y^{\prime}\right\vert ^{2-p}\left(  y^{\prime}-(\delta-\alpha
)y+e^{-\delta(q-1)\tau}\left\vert y\right\vert ^{q-1}y\right)  =0. \label{yss}%
\end{equation}
and system (\ref{sysd}) becomes
\begin{equation}
\left\{
\begin{array}
[c]{c}%
y^{\prime}=\delta y-\left\vert Y\right\vert ^{(2-p)/(p-1)}Y\qquad\qquad
\qquad\qquad\qquad\qquad\qquad\\
Y^{\prime}=(\delta-N)Y-\left\vert Y\right\vert ^{(2-p)/(p-1)}Y+\alpha
y+e^{-\delta(q-1)\tau}\left\vert y\right\vert ^{q-1}y.
\end{array}
\right.  \label{sys}%
\end{equation}
As $\tau\rightarrow\infty,$ this system appears as a perturbation of an
\textit{autonomous} system%
\begin{equation}
\left\{
\begin{array}
[c]{c}%
y^{\prime}=\delta y-\left\vert Y\right\vert ^{(2-p)/(p-1)}Y\qquad\qquad
\qquad\\
Y^{\prime}=(\delta-N)Y-\left\vert Y\right\vert ^{(2-p)/(p-1)}Y+\alpha y
\end{array}
\right.  \label{aut}%
\end{equation}
corresponding to the problem (\ref{hog}). The existence of such a system is
one of the key points of the new results in \cite{Bi1}. If $\delta
(\delta-N)(\delta-\alpha)\leq0,$ it has only one stationnary point $(0,0).$ If
$\delta(\delta-N)(\delta-\alpha)>0,$ which implies $p<2,$ it has three
stationary points:
\begin{equation}
(0,0),\text{\quad}M_{\ell}=(\ell,(\delta\ell)^{p-1}),\text{ and\quad}M_{\ell
}^{\prime}=-M_{\ell}, \label{sta}%
\end{equation}
where $\ell$ is defined at (\ref{ell}). The critical value $\alpha^{\ast}$ of
$\alpha,$ defined at (\ref{eto}) corresponds to the case where the eigenvalues
of the linearized problem at $M_{\ell}$ are imaginary. Observe the relation
\begin{equation}
J_{N}(r)=e^{(N-\delta)\tau}(y(\tau)-Y(\tau)). \label{ret}%
\end{equation}
\medskip

As in \cite{Bi} and \cite{Bi1}, we construct a new energy function, adapted to
system (\ref{sys}), by using the Anderson and Leighton formula for autonomous
systems, see \cite{AnL}. Let%
\begin{equation}
\mathcal{W}(y,Y)=\frac{(2\delta-N)\delta^{p-1}}{p}\left\vert y\right\vert
^{p}+\frac{\left\vert Y\right\vert ^{p^{\prime}}}{p^{\prime}}-\delta
yY+\frac{\alpha-\delta}{2}y^{2}, \label{ww}%
\end{equation}%
\begin{equation}
W(\tau)=\mathcal{W}(y(\tau),Y(\tau))+\frac{1}{q+1}e^{-\delta(q-1)\tau
}\left\vert y(\tau)\right\vert ^{q+1} \label{wt}%
\end{equation}
Then
\begin{equation}
W^{\prime}(\tau)=\mathcal{U}(y(\tau),Y(\tau))-\frac{\delta(q-1)}%
{q+1}e^{-\delta(q-1)\tau}\left\vert y(\tau)\right\vert ^{q+1}, \label{wtp}%
\end{equation}
with%
\begin{equation}
\mathcal{U}(y,Y)=\left(  \delta y-\left\vert Y\right\vert ^{(2-p)/(p-1)}%
Y\right)  \left(  \left\vert \delta y\right\vert )^{p-2}\delta y-Y\right)
(2\delta-N-\mathcal{H}(y,Y)), \label{ut}%
\end{equation}%
\begin{equation}
\mathcal{H}(y,Y)=\left\{
\begin{array}
[c]{c}%
\left(  \delta y-\left\vert Y\right\vert ^{(2-p)/(p-1)}Y\right)  /\left(
\left\vert \delta y\right\vert ^{p-2}\delta y-Y\right)  \quad\quad
\text{if}\quad\left\vert \delta y\right\vert )^{p-2}\delta y\neq Y,\\
\left\vert \delta y\right\vert ^{2-p}/(p-1)\qquad\qquad\qquad\qquad
\qquad\qquad\text{if}\quad\left\vert \delta y\right\vert )^{p-2}\delta y=Y.
\end{array}
\right.  \label{hh}%
\end{equation}
If $2\delta\leq N,$ then $\mathcal{U}(y,Y)\leq0$ on $\mathbb{R}^{2},$ thus $W$
is nonincreasing. If $2\delta\geq N,$ the set
\begin{equation}
\mathcal{L}=\left\{  (y,Y)\in\mathbb{R}^{2}:\mathcal{H}(y,Y)=2\delta
-N\right\}  , \label{curl}%
\end{equation}
is a closed curve surrounding $(0,0),$ symmetric with respect to $(0,0),$ and
bounded, since for any $(y,Y)\in\mathbb{R}^{2},$
\begin{equation}
\mathcal{H}(y,Y)\geq\frac{1}{2}((\delta y)^{2-p}+\left\vert Y\right\vert
^{(2-p)/(p-1)}).\text{ } \label{curi}%
\end{equation}
Introducing the domain $\mathcal{S}$ of $\mathbb{R}^{2}$ with boundary
$\mathcal{L}$ and containing $(0,0),$
\begin{equation}
\mathcal{S}=\left\{  (y,Y)\in\mathbb{R}^{2}:\mathcal{H}(y,Y)<2\delta
-N\right\}  ,\text{ } \label{sos}%
\end{equation}
then $W^{\prime}(\tau)\leq0$ for any $\tau$ such that whenever $(y(\tau
),Y(\tau))\not \in \mathcal{S},$ from (\ref{wtp}).

\subsection{Existence of global solutions}

The first question concerning problem (\ref{un}), (\ref{ini}) is the local
existence and uniqueness near $0.$ It is not straightforward in the case
$p>2,$ and the regularity of the solution differs according to the value of
$p.$ It is shown in \cite{BG} when $p>2$ and $\alpha=\alpha_{0},$ by following
the arguments of \cite{GuV}. We recall and extend the proof to the general case.

\begin{theorem}
\label{exi}For any $a\neq0,$ problem (\ref{un}), (\ref{ini}) admits a unique
solution $w=w(.,a)\in C^{1}\left(  \left[  0,\infty\right)  \right)  ,$ such
that $\left\vert w^{\prime}\right\vert ^{p-2}w^{\prime}\in C^{1}\left(
\left[  0,\infty\right)  \right)  ;$ and
\begin{equation}
\lim_{r\rightarrow0}\left\vert w^{\prime}\right\vert ^{p-2}w^{\prime
}/rw=-(\alpha/N+a^{q+1}); \label{pri}%
\end{equation}
thus $w\in C^{2}\left(  \left[  0,\infty\right)  \right)  $ if $p<2.$ And
$\left\vert w(r)\right\vert \leq a$ on $\left[  0,\infty\right)  .$
\end{theorem}

\begin{proof}
\textbf{Step }$1:$\textbf{ Local existence and uniqueness. }We can
suppose\textbf{ }$a>0.$ Let $\rho>0.$ From (\ref{vir}), any $w\in C^{1}\left(
\left[  0,\rho\right]  \right)  ,$ such that $\left\vert w^{\prime}\right\vert
^{p-2}w^{\prime}\in C^{1}\left(  \left[  0,\rho\right]  \right)  $ solution of
the problem satisfies $w=T(w),$ where
\begin{align}
T(w)(r)  &  =a-%
{\displaystyle\int\nolimits_{0}^{r}}
\left\vert H(w)\right\vert ^{(2-p)/(p-1)}H(w)ds,\nonumber\\
H(w(r))  &  =rw-r^{1-N}J_{N}(r)=rw-r^{1-N}%
{\displaystyle\int\nolimits_{0}^{r}}
s^{N-1}j(w(s))ds, \label{th}%
\end{align}
and $j(r)=(N-\alpha)r-\left\vert r\right\vert ^{q-1}r.$ Reciprocally, the
mapping $T$ is well defined from $C^{0}\left(  \left[  0,\rho\right]  \right)
$ into itself. If $w\in C^{0}\left(  \left[  0,\rho\right]  \right)  $ and
$w=T(w),$ then $w\in C^{1}\left(  \left(  0,\rho\right]  \right)  $ and
$\left\vert w^{\prime}\right\vert ^{p-2}w^{\prime}=H(w),$ hence $\left\vert
w^{\prime}\right\vert ^{p-2}w^{\prime}\in C^{1}\left(  \left(  0,\rho\right]
\right)  $ and $w$ satisfies (\ref{un}) in $\left(  0,\rho\right]  .$ Moreover
$\lim_{r\rightarrow0}j(w(r))=a^{q}-(N-\alpha)a,$ hence $\left\vert w^{\prime
}\right\vert ^{p-2}w^{\prime}(r)=-r((\alpha/N+a^{q-1})+o(1));$ in particular
$\lim_{r\rightarrow0}w^{\prime}(r)=0$, and $\left\vert w^{\prime}\right\vert
^{p-2}w^{\prime}\in C^{1}\left(  \left[  0,\rho\right]  \right)  ,$ and $w$
satisfies (\ref{un}) and (\ref{ini}), and (\ref{pri}) holds. We consider the
ball
\[
\mathcal{B}_{R,M}=\left\{  w\in C^{0}\left(  \left[  0,\rho\right]  \right)
:\left\Vert w-a\right\Vert _{C^{0}\left(  \left[  0,R\right]  \right)  }\leq
M\right\}  ,
\]
where $M$ is a parameter such that $0<M<a/2.$ Notice that $j$ is locally
Lipschitz continuous, since $q>1.$ In case $p<2,$ then the function
$r\mapsto\left\vert r\right\vert ^{(2-p)/(p-1)}r$ has the same property, hence
$T$ is a strict contraction from $\mathcal{B}_{\rho,M}$ into itself for $\rho$
and $M$ small enough. Now suppose $p>2.$ Let $K=K(a,M)$ be the best Lipschitz
constant of $j$ on $\left[  a-M,a+M\right]  .$ For any $w\in\mathcal{B}%
_{R,M},$ and any $r\in\left[  0,\rho\right]  ,$ from (\ref{th})
\begin{equation}
\left(  a-M-\frac{j(a)+MK_{M}}{N}\right)  r\leq H(w(r))\leq\left(
a+M+\frac{-j(a)+MK}{N}\right)  r \label{mm}%
\end{equation}
hence, setting $\mu(a)=a-j(a)/N=(a^{q}+\alpha a)/N>0,$
\[
\mu(a)r/2<H(w(r))<2\mu(a)r
\]
as soon as $M\leq M(a)$ small enough. Then from (\ref{th}),%
\[
\left\Vert T(w)-a\right\Vert _{C^{0}\left(  \left[  0,R\right]  \right)  }%
\leq\left(  2\mu(a)\right)  ^{1/(p-1)}R^{p/(p-1)}%
\]
hence $T(w)\in\mathcal{B}_{\rho,M}$ for $\rho=\rho(a)$ small enough. Now for
any $w_{1},w_{2}\in\mathcal{B}_{\rho,M},$ and any $r\in\left[  0,\rho\right]
,$
\[
\left\vert T(w_{1})(r)-T(w_{2})(r)\right\vert \leq%
{\displaystyle\int\nolimits_{0}^{r}}
\left\vert \left\vert H(w)\right\vert ^{(2-p)/(p-1)}H(w_{1})-\left\vert
H(w_{2})\right\vert ^{(2-p)/(p-1)}H(w)\right\vert (s)ds
\]
and for any $s\in\left[  0,r\right]  ,$ from \cite[p.185]{GuV}, and
\begin{align}
&  \left\vert \left\vert H(w)\right\vert ^{(2-p)/(p-1)}H(w_{1})-\left\vert
H(w_{2})\right\vert ^{(2-p)/(p-1)}H(w)\right\vert (s)\nonumber\\
&  \leq H(w_{2})^{(2-p)/(p-1)}\left\vert H(w_{1})-H(w_{2})\right\vert
(s)\nonumber\\
&  \leq\left(  2\mu(a)\right)  ^{(2-p)/(p-1)}s^{1/(p-1)}\left(  \left\vert
w_{1}-w_{2}\right\vert +Ks^{-N}%
{\displaystyle\int\nolimits_{0}^{s}}
\sigma^{N-1}\left\vert w_{1}-w_{2}\right\vert d\sigma\right) \nonumber\\
&  \leq C(a)s^{1/(p-1)}\left\Vert w_{1}-w_{2}\right\Vert _{C^{0}\left(
\left[  0,R\right]  \right)  } \label{bba}%
\end{align}
with $C(a)=\left(  2\mu(a)\right)  ^{(2-p)/(p-1)}\left(  1+K/N\right)  $
\[
\left\Vert T(w_{1})-T(w_{2})\right\Vert _{C^{0}\left(  \left[  0,R\right]
\right)  }\leq C(a)\rho^{p^{\prime}}\left\Vert w_{1}-w_{2}\right\Vert
_{C^{0}\left(  \left[  0,R\right]  \right)  }\leq\frac{1}{2}\left\Vert
w_{1}-w_{2}\right\Vert _{C^{0}\left(  \left[  0,R\right]  \right)  }%
\]
if $\rho(a)$ is small enough. Then $T$ is a strict contraction from
$\mathcal{B}_{\rho,M}$ into itself. Moreover if $\rho(a)$ and $M(a)$ are small
enough, then for any $b\in\left[  a/2,3a/2\right]  ,$
\[
\left\Vert w(.,b)-w(.,a)\right\Vert _{C^{0}\left(  \left[  0,\rho\right]
\right)  }\leq\left\vert b-a\right\vert +\frac{1}{2}\left\Vert
w(.,a)-w(.,b)\right\Vert _{C^{0}\left(  \left[  0,R\right]  \right)  }%
\]
that means $w(a,.)$ is Lipschitz dependent on $a$ in $\left[  0,\rho
(a)\right]  .$ The same happens for $w^{\prime}(.,a),$ as in (\ref{bba}),
since
\[
\left\vert w^{\prime}(.,b)-w^{\prime}(.,a)\right\vert =\left\vert \left\vert
H(w(.,b))\right\vert ^{(2-p)/(p-1)}H(w(.,b))-\left\vert H(w(.,a))\right\vert
^{(2-p)/(p-1)}H(w(.,a))\right\vert .
\]
\medskip

\noindent\ \textbf{Step }$2:$\textbf{ Global existence and uniqueness. }The
function $w$ on $\left[  0,\rho(a)\right]  $ can be extended on $\left[
0,\infty\right)  .$ Indeed on the definition set,%
\begin{equation}
E(r)=\frac{1}{p^{\prime}}\left\vert w^{\prime}\right\vert ^{p}+\frac{\alpha
}{2}w^{2}+\frac{1}{q+1}\left\vert w\right\vert ^{q+1}\leq E(0)=\frac{\alpha
}{2}a^{2}+a^{q+1}, \label{mn}%
\end{equation}
hence $w$ and $w^{\prime}$ stay bounded, and $\left\vert w(r)\right\vert \leq
a$ on $\left[  0,\infty\right)  $. The extended function is unique. Indeed
existence and uniqueness hold near at any point $r_{1}>0$ such that
$w^{\prime}(r_{1})\neq0$ or $p\leq2$ from the Cauchy-Lipschitz theorem; if
$w^{\prime}(r_{1})=0,w(r_{1})\neq0$ and $p>2,$ it follows from fixed point
theorem as above; finally if $w(r_{1})=w^{\prime}(r_{1})=0$, then $w\equiv0$
on $\left[  r_{1},\infty\right)  $ since $E$ is nonincreasing.
\end{proof}

\begin{remark}
\label{dep} For any $r_{1}\geq0,$ we have a local continuous dependence of $w$
and $w^{\prime}$ in function of $c_{1}=w(r_{1})$ and $c_{2}=$ $w^{\prime
}(r_{1}).$ Indeed the only delicate case is $c_{1}=c_{2}=0.$ Since $E$ is
nonincreasing, then for any $\varepsilon>0,$, if $\left\vert w(r_{1}%
)\right\vert +\left\vert w^{\prime}(r_{1})\right\vert \leq\varepsilon,$ then
$\sup_{\left[  r_{1},\infty\right)  }\left\vert w(r)\right\vert +\left\vert
w^{\prime}(r)\right\vert \leq C(\varepsilon),$ where $C$ is continuous; thus
the dependence holds on whole $\left[  r_{1},\infty\right)  $. In particular,
for any $a\in$ $\mathbb{R}$, $w(.,a)$ and $w^{\prime}(.,a)$ depend
continuously on $a$ on any segment $\left[  0,R\right]  .$ If for some
$a_{0},$ $w(.,a_{0})$ has a compact support, the dependance is continuous on
$\mathbb{R}.$ As a consequence, $w(.,.)$ and $w^{\prime}(.,.)\in C^{0}\left(
\left[  0,\infty\right)  \times\mathbb{R}\right)  .$
\end{remark}

\begin{remark}
\label{maxi}Any local solution $w$ of problem (\ref{un}) near a point
$r_{1}>0$ is defined on a maximal interval $(R_{w},\infty)$ with $0\leq
R_{w}<r_{1}.$
\end{remark}

\subsection{First oscillatory properties}

Let us begin by simple remarks on the behaviour of the solutions.

\begin{proposition}
\label{pro} Let $w$ be any solution of problem (\ref{un}). Then
\begin{equation}
\lim_{r\rightarrow\infty}w(r)=0,\qquad\lim_{r\rightarrow\infty}w^{\prime
}(r)=0. \label{wwp}%
\end{equation}
If $w>0$ for large $r$, then $w^{\prime}<0$ for large $r.$
\end{proposition}

\begin{proof}
Let $w$ be any solution on $\left[  r_{0},\infty\right)  $, $r_{0}>0.$ Since
function $E$ is nonincreasing, $w$ and $w^{\prime}$ are bounded, and $E$ has a
finite limit $\xi\geq0.$ Consider the function $V=V_{\lambda,d,e}$ defined at
(\ref{vla}) with $\lambda=0,$ $\sigma=(N-1)/2,$ $e=\alpha+\sigma.$ It is
bounded near $\infty$ and satisfies%

\begin{align*}
-rV^{\prime}(r)  &  =\frac{N-1}{2}(\left\vert w^{\prime}\right\vert
^{p}+\left\vert w\right\vert ^{q+1}+\alpha w^{2}+\frac{N}{2}r^{-1}w\left\vert
w^{\prime}\right\vert ^{p-2}w^{\prime}+r^{2}w^{\prime^{2}})\\
&  \geq\frac{N-1}{2}E(r)+o(1)\geq\frac{N-1}{2}\xi+o(1).
\end{align*}
If $\xi>0,$ then $V$ is not integrable, which is contradictory. Thus $\xi=0$
and (\ref{wwp}) holds. Moreover at each extremal point $r$ such that $w(r)>0,$
from
\begin{equation}
(\left\vert w^{\prime}\right\vert ^{p-2}w^{\prime})^{\prime}(r)=-(\alpha
+w(r)^{q+1})w(r), \label{ext}%
\end{equation}
thus $r$ is unique and it is a maximum. If $w(r)>0$ for large $r,$ then from
(\ref{wwp}) necessarily $w^{\prime}<0$ for large $r$.\medskip
\end{proof}

Now we give some first results concerning the possible zeros of the solutions.
If $p<2$ then any solution $w\not \equiv 0$ of (\ref{un}) has only isolated
zeros, from the Cauchy-Lipschitz theorem. On the contrary if $p>2,$ it can
exist $r_{1}>0$ such that $w(r_{1})=w^{\prime}(r_{1})=0,$ and then from
uniqueness $w\equiv0$ on $\left[  r_{1},\infty\right)  .$

\begin{proposition}
\label{zer} (i) Assume $\alpha<N.$ Let $\underline{a}=(N-\alpha)^{1/(q-1)}.$
Then for any $a\in\left(  0,\underline{a}\right]  $, $w(r,a)>0$ on $\left[
0,\infty\right)  .$

\noindent(ii) Assume $p_{1}<p$ and $N\leq\alpha.$ Then for any $a\neq0,$
$w(r,a)$ has at least one isolated zero.

\noindent(iii) Assume $p<2.$ Then for any $0<m<M<\infty,$ any solution $w$ of
(\ref{un}) has a finite number of zeros in $\left[  m,M\right]  ,$ or
$w\equiv0$ in $\left[  m,M\right]  .$

\noindent(iv) Assume $p>2$ or $\alpha<\max(N,\eta).$ Then for any $m>0,$ any
solution $w$ of problem (\ref{un}) $w$ has a finite number of isolated zeros
in $\left[  m,\infty\right)  ,$ or $w\equiv0$ in $\left[  m,\infty\right)  $.
\end{proposition}

\begin{proof}
(i) Let $a\in\left(  0,\underline{a}\right]  .$ Assume that there exists a
first $r_{1}>0$ such that $w(r_{1},a)=0,$ hence $w^{\prime}(r_{1},a)\leq0.$
Let us consider $J_{N}$ defined by (\ref{gg}).\ Then $J_{N}^{\prime}(r)\geq0$
on $\left[  0,r_{1}\right)  $, since $0\leq w(r)\leq a,$ and $J_{N}(0)=0,$ and
$J_{N}(r_{1})=r_{1}^{N-1}\left\vert w^{\prime}(r_{1})\right\vert
^{p-2}w^{\prime}(r_{1})\leq0,$ thus $J_{N}^{\prime}\equiv0$ on $\left[
0,r_{1}\right]  $, thus $w\equiv\underline{a}$, which contradicts
(\ref{un}).\medskip

\noindent(ii) Suppose that for some $a>0,$ $w(r)=w(r,a)>0$ on $\left[
0,\infty\right)  .$ Since $N\leq\alpha,$ there holds $J_{N}^{\prime}(r)<0$ on
$\left[  0,\infty\right)  ,$ and $J_{N}(0)=0,$ hence $J_{N}(r)\leq0.$ Then
$r\longmapsto r^{p^{\prime}}-\delta w^{-\delta}$ is nonincreasing.

$\bullet$ If $p>2,$ it is impossible, thus $w$ has a first zero $r_{1},$ and
$J_{N}^{\prime}(r)<0$ on $\left[  0,r_{1}\right)  ,$ thus $J_{N}(r_{1})<0,$
then $w^{\prime}(r_{1})<0$ and $r_{1}$ is isolated.

$\bullet$ If $p<2,$ there exists $c>0$ such that for large $r,$ $J_{N}%
(r)\leq-c,$ hence $w(r)+cr^{-N}\leq\left\vert w^{\prime}(r)\right\vert
^{p-1}/r.$ Then there exists another $c>0$ such that $w^{\prime}%
+cr^{(1-N)/(p-1)}\leq0.$ If $N=1$ it contradicts Proposition \ref{pro}. If
$2\leq N,$ then $p<N,$ and $w-cr^{-\eta}/\eta$ decreases to $0$, thus
$\delta\leq\eta,$ which contradicts $N<\delta$, which means $p_{1}<p,$ from
(\ref{dn}).\medskip

\noindent(iii) Suppose that $w$ has an infinity of isolated zeros in $\left[
m,M\right]  .$ Then there exists a sequence of zeros converging to some
$\overline{r}\in\left[  m,M\right]  .$ We can extract an increasing (or a
decreasing) subsequence of zeros $\left(  r_{n}\right)  $ such that $w>0$ on
$\left(  r_{2n},r_{2n+1}\right)  $ and $w<0$ on $\left(  r_{2n-1}%
,r_{2n}\right)  .$ There exists $s_{n}\in\left(  r_{n},r_{n+1}\right)  $ such
that $w^{\prime}(s_{n})=0;$ since $w\in C^{1}\left[  0,\infty\right)  ,$ it
implies $w(\overline{r})=w^{\prime}(\overline{r})=0.$ It is impossible because
$p<2.$\medskip

\noindent(iv) Suppose that $w\not \equiv 0$ in $\left[  m,\infty\right)  $.
Let $Z$ be the set of its isolated zeros in $\left[  m,\infty\right)  $.
Notice that $m$ is not an accumulation point of $Z,$ since $(w(m),w^{\prime
}(m))\neq(0,0).$ Let $\rho_{1}<$ $\rho_{2},$ be two consecutive zeros, thus
such that $\rho_{1}$ is isolated, and $\left\vert w\right\vert >0$ on $\left(
\rho_{1},\rho_{2}\right)  .$ We make the substitution (\ref{cge}), where $d>0$
will be choosen after. At each point $\tau$ such that $y_{d}^{\prime}%
(\tau)=0,$ and $y_{d}(\tau)\neq0,$ we deduce
\begin{equation}
(p-1)y_{d}^{\prime\prime}=y_{d}\left(  (p-1)d(\eta-d)+e^{((p-2)d+p)\tau
}\left\vert dy_{d}\right\vert ^{2-p}\left(  d-\alpha-e^{-d(q-1)\tau}\left\vert
y_{d}\right\vert ^{q-1}y_{d}\right)  \right)  ; \label{teg}%
\end{equation}
if $\tau\in\left(  e^{\rho_{1}},e^{\rho_{2}}\right)  $ is an maximal point of
$\left\vert y_{d}\right\vert $, it follows that%
\begin{equation}
e^{((p-2)d+p)\tau}\left\vert dy_{d}(\tau)\right\vert ^{2-p}\left(
d-\alpha-e^{-d(q-1)\tau}\left\vert y_{d}(\tau)\right\vert ^{q-1}\right)
\leq(p-1)d(d-\eta) \label{tig}%
\end{equation}
Setting $\rho=e^{\tau}\in\left(  \rho_{1},\rho_{2}\right)  ,$ it means
\begin{equation}
\rho^{p}\left\vert w(\rho)\right\vert ^{2-p}\left(  d-\alpha-\left\vert
w\right\vert ^{q-1}(\rho)\right)  \leq(p-1)d^{p-1}(d-\eta). \label{tog}%
\end{equation}
If $p>2,$ we fix $d>\alpha.$ Since $\lim_{r\rightarrow\infty}w(r)=0,$ the
coefficient of $\rho^{p}$ in the left-hand side tends to $\infty$ as
$\rho\rightarrow\infty,$ hence $\rho$ is bounded, hence also $\rho_{1},$ thus
$Z$ is bounded. If $\alpha<\eta,$ we take $d\in\left(  \alpha,\eta\right)  .$
Then the right hand side is negative, and the left hand side is nonnegative
for large $r,$ hence again $Z$ is bounded. If\textbf{ } $\alpha<N,$ we use
function $J_{N}:$
\begin{equation}
J_{N}(\rho_{2})-J_{N}(\rho_{1})=\rho_{2}^{N-1}\left\vert w^{\prime}\right\vert
^{p-2}w^{\prime}(\rho_{2})-\rho_{1}^{N-1}\left\vert w^{\prime}\right\vert
^{p-2}w^{\prime}(\rho_{1})=%
{\displaystyle\int\nolimits_{\rho_{1}}^{\rho_{2}}}
s^{N-1}w(N-\alpha-\left\vert w\right\vert ^{q-1}w)ds \label{tug}%
\end{equation}
and the integral has the sign of $w$ for large $\rho,$ hence a contradiction.
In any case $Z$ is bounded. Suppose that $Z$ is infinite; then $p>2$ from step
(iii), and there exists a sequence of zeros $\left(  r_{n}\right)  $,
converging to some $\overline{r}\in\left(  m,\infty\right)  $ such that
$w(\overline{r})=w^{\prime}(\overline{r})=0$. Then there exists a sequence
$\left(  \tau_{n}\right)  $ of maximal points of $\left\vert y_{d}\right\vert
$ converging to $\overline{\tau}=\ln\overline{r}$. Taking $\rho=\rho
_{n}=e^{\tau_{n}}$ in (\ref{tog}) leads to a contradiction, since the
left-hand side tends to $\infty.$\medskip
\end{proof}

When $w$ has a constant sign for large $r,$ we can give some informations on
the behaviour for large $\tau$ of the solutions $(y,Y)$ of system (\ref{sys}),
in particular the convergence to a stationary point of the autonomous system
(\ref{aut}): We have also a majorization in one case when the solution is
changing sign.

\begin{lemma}
\label{com} Let $w$ be any solution of (\ref{un}), and $(y,Y)$ be defined by
(\ref{cha}).

(i) If $y>0$ and $y$ is not monotone for large $\tau$, then $Y$ is not
monotone for large $\tau,$ and either $\max(\alpha,N)<\delta$ and $\lim
_{\tau\rightarrow\infty}y(\tau)=\ell,$ or $\delta<\min(\alpha,N)$ and
$\lim\inf_{\tau\rightarrow\infty}y(\tau)\leq\ell\leq\lim\sup_{\tau
\rightarrow\infty}y(\tau).$

(ii) If $y$ $>0$ and $y$ has a limit $l$ at $\infty,$ then either $l=0$ and
$\lim_{\tau\rightarrow\infty}Y(\tau)=0,$ or $(\delta-N)(\delta-\alpha)>0$ and
$l=\ell$ and $\lim_{\tau\rightarrow\infty}(y(\tau),Y(\tau))=M_{\ell},$ or
$\delta=\alpha=N$ and $\lim_{\tau\rightarrow\infty}Y(\tau)=(\delta l)^{p-1}.$

(iii) If $y>0$ and $y$ is nondecreasing for large $\tau$ and $\lim
_{\tau\rightarrow\infty}y(\tau)=\infty,$ then $\lim_{\tau\rightarrow\infty
}Y(\tau)=\infty.$

(iv) If $y$ is changing sign for large $\tau$ (which implies $p<2$) and
$\alpha<\delta,$ then $N<\delta$ and $\left\vert y(\tau)\right\vert \leq
\ell\left(  1+o(1)\right)  $ and $\left\vert Y(\tau)\right\vert \leq
(\delta\ell)^{p-1}(1+o(1))$ near $\infty.$
\end{lemma}

\begin{proof}
From Proposition \ref{pro}, $Y(\tau)>0$ for large $\tau$ in cases (i) to
(iii).\medskip

\noindent(i) Suppose that $y$ is not monotone near $\infty.$ Then there exists
an increasing sequence $\left(  \tau_{n}\right)  $ such that $\tau
_{n}\rightarrow\infty,$ $y^{\prime}(\tau_{n})=0,$ $y^{\prime\prime}(\tau
_{2n})\geq0,$ $y^{\prime\prime}(\tau_{2n+1})\leq0,$ $y(\tau_{2n})\leq
y(\tau)\leq y(\tau_{2n+1})$ on $\left(  \tau_{2n},\tau_{2n+1}\right)  ,$
$y(\tau_{2n})\leq y(\tau)\leq y(\tau_{2n-1})$ on $\left(  \tau_{2n-1}%
,\tau_{2n}\right)  ,$ and $y(\tau_{2n})<y(\tau_{2n+1}).$From (\ref{yss}),%
\begin{equation}
(p-1)y^{\prime\prime}(\tau_{n})=\delta^{2-p}y(\tau_{n})\left(  y(\tau
_{n})^{2-p}\left(  \delta-\alpha-e^{-\delta(q-1)\tau_{n}}y(\tau_{n}%
)^{q-1})\right)  -(\delta-N)\delta^{p-1}\right)  \label{pow}%
\end{equation}
From Proposition \ref{pro}, $e^{-\delta\tau}y(\tau)=o(1)$ near $\infty$ and
\begin{align*}
&  y(\tau_{2n+1})^{2-p}\left(  \alpha-\delta+e^{-\delta(q-1)\tau_{2n+1}}%
y(\tau_{2n+1})^{q-1})\right) \\
&  >(N-\delta)\delta^{p-1}\geq y(\tau_{2n})^{2-p}\left(  \alpha-\delta
+e^{-\delta(q-1)\tau_{2n}}y(\tau_{2n})^{q-1}\right)  >y(\tau_{2n}%
)^{2-p}\left(  \alpha-\delta\right)  .
\end{align*}
Then either $\alpha<\delta$ and $N<\delta$ and $\ell\leq y(\tau_{2n})\leq
y(\tau_{2n+1})\leq\ell(1+o(1)),$ hence $\lim_{\tau\rightarrow\infty}%
y(\tau)=\ell.$ Or $\delta<\alpha$ and $\delta<N,$ and $y(\tau_{2n})<\ell,$ and
$\ell\leq y(\tau_{2n+1})(1+o(1))$. If $Y$ is monotone near $\infty,$ then from
(\ref{sys}), $y^{\prime\prime}=\delta y^{\prime}-Y^{(2-p)/(p-1)}Y^{\prime},$
hence $e^{-\delta t}y^{\prime}$ is monotone, which contradicts the existence
of a sequence $\left(  \tau_{n}\right)  $ as above. Thus $Y$ is not
monotone.\medskip

\noindent(ii) Let $l=\lim_{\tau\rightarrow\infty}y\geq0.$ If $Y$ is monotone,
either $\lim_{\tau\rightarrow\infty}Y=\infty,$ which is impossible, since then
$y^{\prime}\rightarrow-\infty;$ or $Y$ has a finite limit $\lambda\geq0.$ If
$Y$ is not monotone, at the extremal points $\tau$ of $Y,$ we have
\[
\left\vert Y(\tau)\right\vert ^{(2-p)/(p-1)}Y(\tau)-(\delta-N)Y(\tau)=\alpha
l+o(1),
\]
from (\ref{sys}), thus $Y$ has a limit at these points, hence $Y$ still has a
limit $\lambda$. From (\ref{sys}), $y^{\prime}$ has a limit, necessarily $0,$
hence $\lambda=(\delta l)^{p-1}.$Then $Y^{\prime}$ has a limit, necessarily
$0,$ and $(\delta-N)(\delta l)^{p-1}=(\delta-\alpha)l;$ thus $l=0=\lambda,$ or
$(\delta-N)(\delta-\alpha)>0$ and $l=\ell,$ $\lambda=(\delta\ell)^{p-1},$ or
$\delta=\alpha=N.$\medskip

\noindent(iii) Suppose that $y$ is nondecreasing and $\lim_{\tau
\rightarrow\infty}y(\tau)=\infty.$ Then either $Y$ is not monotone, and at
minimum points it tends to $\infty$ from (\ref{sys}), then $\lim
_{\tau\rightarrow\infty}Y(\tau)=\infty$. 0r $Y$ is monotone; if it has a
finite limit, then $\lim_{\tau\rightarrow\infty}Y^{\prime}(\tau)=\infty$ from
(\ref{sys}), which is impossible. Then again $\lim_{\tau\rightarrow\infty
}Y(\tau)=\infty.$\medskip

\noindent(iv) Assume that $y$ does not keep a constant sign near $\infty;$
then also $w,$ thus also $w^{\prime},$ and in turn $Y.$ At any maximal point
$\theta$ of $\left\vert y\right\vert $, one finds
\[
(p-1)y^{\prime\prime}(\theta)=\delta^{2-p}y(\theta)\left(  \left\vert
y(\theta)\right\vert ^{2-p}\left(  \delta-\alpha-e^{-\delta(q-1)\theta
}\left\vert y(\theta)\right\vert ^{q-1}\right)  -(\delta-N)\delta
^{p-1}\right)  ,
\]
hence
\[
\left\vert y(\theta)\right\vert ^{2-p}\left(  \delta-\alpha+o(1)\right)
\leq(\delta-N)\delta^{p-1}.
\]
Since $\delta-\alpha>0,$ it follows that $\delta-N>0$ and $\left\vert
y(\tau)\right\vert \leq\ell(1+o(1))$ near $\infty$. Similarly at any maximal
point $\vartheta$ of $\left\vert Y\right\vert ,$ one finds
\begin{align*}
Y^{\prime\prime}(\vartheta)  &  =(\alpha+e^{-\delta(q-1)\vartheta}\left\vert
y(\vartheta)\right\vert ^{q-1})y^{\prime}+\delta(q-1)e^{-\delta(q-1)\vartheta
}\left\vert y(\vartheta)\right\vert ^{q-1}y\\
0  &  =(\delta-N)Y(\vartheta)-\left\vert Y(\vartheta)\right\vert
^{(2-p)/(p-1)}Y(\vartheta)+(\alpha+e^{-\delta(q-1)\vartheta}\left\vert
y(\vartheta)\right\vert ^{q-1})y(\vartheta)
\end{align*}
which implies
\[
\left\vert Y(\vartheta)\right\vert ^{(2-p)/(p-1)}\left(  \delta-\alpha
+o(1)\right)  \leq(\delta-N)\delta
\]
thus $\left\vert Y(\tau)\right\vert \leq(\delta\ell)^{p-1}(1+o(1))$ near
$\infty.$
\end{proof}

\subsection{Further results by blow up techniques}

Next we give two results obtained by rescaling and blow up techniques. The
first one consists in a scaling leading to the equation
\begin{equation}
r^{1-N}\left(  r^{N-1}\left\vert v^{\prime}\right\vert ^{p-2}v^{\prime
}\right)  ^{\prime}+\left\vert v\right\vert ^{q-1}v=0. \label{vq}%
\end{equation}
without term in $rw^{\prime},$ extending the result of (\cite[Proposition
3.4]{W1}) to the case $p\neq2.$ It gives a result in the subcritical case
$q<q^{\ast},$ and does not depend on the value of $\alpha.$

\begin{proposition}
\label{sig} Assume that $1<q<q^{\ast}($thus $p>p_{2})$. Then for any
$m\in\mathbb{N},$ there exists $\overline{a_{m}}$ such that for any
$a>\overline{a_{m}},$ $w(.,a)$ admits at least $m+1$ isolated zeros. And for
fixed $m,$ the $m^{th}$ zero of $w(.,a)$ tends to $0$ as $a$ tends to
$\infty.$
\end{proposition}

\begin{proof}
(i) First we show that there exists $a_{\ast}>0,$ such that for any
$a>a_{\ast},$ $w(.,a)$ cannot stay positive on $\left[  0,\infty\right)  $.
Suppose that there exists $\left(  a_{n}\right)  $ tending to $\infty,$ such
that $w_{n}(r)=w(r,a_{n})\geq0$ on $\left[  0,\infty\right)  $, and let
\begin{equation}
v_{n}(r)=a_{n}^{-1}w_{n}(a_{n}^{-1/\alpha_{0}}r). \label{cna}%
\end{equation}
Then $v_{n}(0)=1,$ $v_{n}^{\prime}(0)=0$ and $v_{n}$ satisfies the equation%
\begin{equation}
\left(  r^{N}(a_{n}^{1-q}v_{n}+r^{-1}\left\vert v_{n}^{\prime}\right\vert
^{p-2}v_{n}^{\prime})\right)  ^{\prime}+r^{N-1}\left(  (\alpha-N)a_{n}%
^{1-q}v_{n}+\left\vert v_{n}\right\vert ^{q-1}v_{n}\right)  =0. \label{vvo}%
\end{equation}
$\ $From (\ref{mn}) applied to $w_{n}$
\[
v_{n}(r)\leq1,\qquad\left\vert v_{n}^{\prime}(r)\right\vert ^{p}\leq
p^{\prime}\left(  \frac{\alpha}{2}a_{n}^{1-q}+\frac{1}{q+1}\right)  \text{
}\qquad\text{in }\left[  0,\infty\right)  ,
\]
thus $v_{n}$ and $v_{n}^{\prime}$ are uniformly bounded in $\left[
0,\infty\right)  .$ If $p\leq2,$ then $v_{n}^{\prime\prime}$ is uniformly
bounded on any compact $\mathcal{K}$ of $\left(  0,\infty\right)  ,$ from
(\ref{un}), and up to a diagonal sequence, $v_{n}$ converges uniformly in
$C_{loc}^{1}\left(  0,\infty\right)  $ to a function $v$. If $p>2,$ then, from
(\ref{vvo}), the derivatives of $r^{N}(a_{n}^{1-q}v_{n}+\left\vert
v_{n}^{\prime}\right\vert ^{p-2}v_{n}^{\prime})$ are uniformly bounded on any
$\mathcal{K}$, and $a_{n}^{1-q}v_{n}$ converges unifomly to $0$ in $\left[
0,\infty\right)  ,$ and up to a diagonal sequence, $\left\vert v_{n}^{\prime
}\right\vert ^{p-2}v_{n}^{\prime}$ converges uniformly on any $\mathcal{K}$,
hence also $v_{n}^{\prime},$ thus $v_{n}$ converges uniformly in $C_{loc}%
^{1}\left(  0,\infty\right)  $ to a nonnegative function $v\in C^{1}\left(
0,\infty\right)  .$ For any $r>0,$
\[
\left\vert v_{n}^{\prime}\right\vert ^{p-2}v_{n}^{\prime}(r)=-a_{n}%
^{1-q}rv_{n}(r)+r^{1-N}%
{\displaystyle\int\nolimits_{0}^{r}}
s^{N-1}\left(  a_{n}^{1-q}(N-\alpha)v_{n}-\left\vert v_{n}\right\vert
^{q-1}v_{n}\right)  ds,
\]
hence
\begin{equation}
\left\vert v^{\prime}\right\vert ^{p-2}v^{\prime}(r)=-r^{1-N}%
{\displaystyle\int\nolimits_{0}^{r}}
s^{N-1}\left\vert v\right\vert ^{q-1}vds\qquad\text{in }\left(  0,\infty
\right)  . \label{vit}%
\end{equation}
In particular $v^{\prime}(r)\rightarrow0$ as $r\rightarrow0,$ hence $v$ can be
extended in a function in $C^{1}(\left[  0,\infty\right)  ),$ such that
$v(0)=1,$ and $v^{\prime}(r)<0.$ Using the form (\ref{un}) for the equation in
$v_{n},$ $v_{n}^{\prime\prime}$ converges uniformly on any $\mathcal{K},$
hence $v\in C^{2}\left(  0,\infty\right)  \cap$ $C^{1}(\left[  0,\infty
\right)  )$ and is solution of the equation (\ref{vq}) such that $v(0)=1$ and
$v^{\prime}0)=0$. But this equation has no nonnegative solution except $0$
since $q<q^{\ast}.$ Moreover the zeros of function $v$ are all isolated, and
form a sequence $\left(  r_{n}\right)  $ tending to $\infty,$ see \cite{Bi},
\cite{BiPo} and \cite{SZ}. Then we reach a contradiction.\medskip

\noindent(ii) Now let $m\geq0.$ As in \cite[Proposition 3.4]{W1}, assume that
there exists a sequence $\left(  a_{n}\right)  $ tending to $\infty,$ such
that $w_{n}(r)=w(r,a_{n})$ has at most $m$ isolated zeros, hence also $v_{n}$.
Up to a subsequence we can suppose that all the $v_{n}(r)$ have the same
number of isolated zeros $\overline{m}:r_{0,n},r_{1,n},..,r_{\overline{m},n}.$
Let $M>0$ such that $r_{0},r_{1},..,r_{\overline{m}}\in\left(  0,M\right)  .$
Then for $n$ large enough, $r_{0,n},r_{1,n},..,r_{\overline{m},n}\in\left(
0,M+1\right)  .$ Either $v_{n}(r)$ has no zero on $\left[  M+1,\infty\right)
,$ or there is a unique zero $r_{\overline{m},n+1}$ such that $v_{n}(r)$ has a
compact support $\left[  0,r_{\overline{m},n+1}\right]  .$ Up to a
subsequence, all the $v_{n}$ are nonnegative or nonpositive on $\left[
M+1,\infty\right)  ;$ then the same holds for $v,$ and we get a contradiction.
Thus for $a$ large enough, $w(.,a)$ has at least $m+1$ zeros. Moreover, as in
\cite{W1}, the $m$ first zeros stay in a compact set, and from (\ref{cna}) the
$m^{th}$ zero of $w(.,a)$ tends to $0$ as $a\rightarrow\infty.\medskip$
\end{proof}

Now we make a scaling leading to the problem without source
\begin{equation}
r^{1-N}\left(  r^{N-1}\left\vert v^{\prime}\right\vert ^{p-2}v^{\prime
}\right)  ^{\prime}+rv^{\prime}+\alpha v=0. \label{hg}%
\end{equation}
It gives informations when the regular solutions of (\ref{hg}) are changing
sign, in particular $p_{2}<p<2,$ and $\delta<\alpha$. It does not depend on
the value of $q.$

\begin{proposition}
Assume that $p_{2}<p<2,$ $\delta<\alpha.$ Then there exists an $\alpha_{c}%
\in\left(  \eta,\alpha^{\ast}\right)  $ such that if $\alpha>\alpha_{c},$ then
for any $m\in\mathbb{N},$ there exists $\overline{a_{m}}$ such that for any
$0<a<\overline{a_{m}},$ $w(.,a)$ admits at least $m+1$ isolated zeros. And for
fixed $m,$ the $m^{th}$ zero of $w(.,a)$ tends to $0$ as $a$ tends to
$\infty.$
\end{proposition}

\begin{proof}
Suppose that there exists $\left(  a_{n}\right)  $ tending to $0,$ such that
$w_{n}(r)=w(r,a_{n})\geq0$ on $\left[  0,\infty\right)  $, and let
\[
v_{n}(r)=a_{n}^{-1}w_{n}(a_{n}^{-1/\delta}r).
\]
Then $v_{n}(0)=1,$ $v_{n}^{\prime}(0)=0$ and $v_{n}$ satisfies equation
\[
\left(  r^{N}(v_{n}+r^{-1}\left\vert v_{n}^{\prime}\right\vert ^{p-2}%
v_{n}^{\prime})\right)  ^{\prime}+r^{N-1}\left(  (\alpha-N)v_{n}+a_{n}%
^{q-1}\left\vert v_{n}\right\vert ^{q-1}v_{n}\right)  =0,
\]
and estimates%
\[
v_{n}(r)\leq1,\qquad\left\vert v_{n}^{\prime}(r)\right\vert ^{p}\leq
p^{\prime}\left(  \frac{\alpha}{2}+\frac{a_{n}^{q-1}}{q+1}\right)
\qquad\text{in }\left[  0,\infty\right)  .
\]
As above we construct a solution $v\in C^{2}\left(  0,\infty\right)  \cap$
$C^{1}(\left[  0,\infty\right)  )$ of the equation (\ref{hg}). But from
\cite{Bi1}, there exists $\alpha_{c}\in\left(  \eta,\alpha^{\ast}\right)  $
such that the regular solutions of (\ref{hg}) are oscillating for
$\alpha>\alpha_{c}$, hence we conclude as above.
\end{proof}

\begin{remark}
This scaling does not give any result when the regular solutions of (\ref{hg})
have a constant sign: it is the case for example when $\alpha=N:$ they are the
Barenblatt solutions, they have a compact support when $p>2$ and a behaviour
in $r^{-\delta}$ near $\infty$ when $p<2.$ Nevertheless if $p>p_{1},$ all the
solutions $w(.,a)$ of (\ref{un}) have at least one zero, from Proposition
\ref{zer}.
\end{remark}

\subsection{Upper estimates of the solutions}

Here we get the behaviour at infinity for solutions \textit{of any sign}. We
extend the results of \cite{HW} obtained for $p=2,$ giving upper estimates
with continous dependence, which also improve the results of \cite{Qi}:

\begin{proposition}
\label{der} Let $d\geq0.$

\noindent(i) Assume that the solution $w$ of problem (\ref{un}), (\ref{ini})
satisfies
\begin{equation}
\left\vert w(r)\right\vert \leq C_{d}(1+r)^{-d}, \label{cor}%
\end{equation}
on $\left[  0,\infty\right)  ,$ for some $C_{d}>0,$ then there exists another
$C_{d}^{\prime}>0,$ depending continuously on $C_{d},$ such that%
\begin{equation}
\left\vert w^{\prime}(r)\right\vert \leq C_{d}^{\prime}(1+r)^{-d-1}.
\label{cer}%
\end{equation}

\noindent(ii) For any solution of (\ref{un}) such that $w(r)=O(r^{-d})$ near
$\infty,$ then $w^{\prime}(r)=O(r^{-d-1})$ near $\infty.$
\end{proposition}

\begin{proof}
(i) We can assume that $w\not \equiv 0.$ Let $r\geq R\geq0;$ we set
\begin{equation}
f_{R}(r)=\exp\left(  \frac{1}{p-1}%
{\displaystyle\int\nolimits_{R}^{r}}
s\left\vert w^{\prime}\right\vert ^{2-p}ds\right)  . \label{fro}%
\end{equation}
The function is well defined when $p<2$ from (\ref{pri}), and $f_{R}\in
C^{1}(\left[  R,\infty\right)  )$. When $p>2,$ from Proposition \ref{zer},
(iv), the function $w$ has a finite number of isolated zeros and either there
exists a first $\bar{r}>0$ such that $w(\bar{r})=w^{\prime}(\bar{r})=0,$ or
$w$ has no zero for large $r,$ and we set $\bar{r}=\infty.$ In the last case
case, from Proposition \ref{pro}, the set of zeros of $w^{\prime}$ is bounded.
If $w^{\prime}(\tilde{r})=0$ for some $\tilde{r}\in\left(  0,\bar{r}\right)
,$ then, from (\ref{un}), $(\left\vert w^{\prime}\right\vert ^{p-2}w^{\prime
})^{\prime}$ has a nonzero limit $\lambda$ at $\tilde{r}$, hence $\tilde{r}$
is an isolated zero of $w$ and
\[
\left\vert w^{\prime}(s)\right\vert ^{2-p}=\left\vert \lambda\right\vert
^{(2-p)/(p-1)}(s-\tilde{r})^{-1+1/(p-1)}(1+o(1))
\]
near $\tilde{r}$. Then $s\left\vert w^{\prime}\right\vert ^{2-p}\in
L_{loc}^{1}\left(  R,\infty\right)  $, thus $f_{R}$ is absolutely continuous
on $\left[  R,\bar{r}\right)  $ if $\bar{r}=\infty.$ Let $k=k(N,p,d)>0$ be a
parameter, such that $K=k-(N-1)/(p-1)>0$, and $k>1+d.$ By computation, for
almost any $r\in\left(  R,\bar{r}\right)  ,$%
\[
\left(  r^{k}f_{R}(w^{\prime}-Kr^{-1}w\right)  ^{\prime}=-K(k-1)r^{k-2}%
f_{R}w-r^{k-1}f_{R}^{\prime}w(\alpha+K+\left\vert w\right\vert ^{q-1})
\]
hence for any $r\in\left[  R,\bar{r}\right)  ,$
\begin{align}
r^{k}f_{R}w^{\prime}  &  =R^{k-1}(Rw^{\prime}(R)-Kw(R))+Kr^{k-1}%
f_{R}w\nonumber\\
&  -K(k-1)%
{\displaystyle\int\limits_{R}^{r}}
s^{k-2}f_{R}wds-%
{\displaystyle\int\limits_{R}^{r}}
s^{k-1}f_{R}^{\prime}w(\alpha+K+\left\vert w\right\vert ^{q-1})ds.
\label{plot}%
\end{align}
Assume (\ref{cor}) and take $R=0,$ and divide by $f_{0}.$ From our choice of
$k$, and since $f^{\prime}\geq0,$ we obtain
\[
r^{k}\left\vert w^{\prime}(r)\right\vert \leq\tilde{C}_{d}r^{k-1-d}%
\]
on $\left[  0,\bar{r}\right)  $ and then on $\left[  0,\infty\right)  ,$ where
$\tilde{C}_{d}=C_{d}(K+K(k-1)/(k-1-d)+\alpha+K)+C_{d}^{q-1},$ and
$K=K(N,p,d);$ this holds in particular on $\left[  1,\infty\right)  ;$ on
$\left[  0,1\right]  ,$ from (\ref{mn}),
\[
\left\vert w^{\prime}(r)\right\vert \leq p^{\prime}(\alpha C_{d}/2+C_{d}%
^{q-1}),
\]
and (\ref{cer}) holds.\medskip

\noindent(ii) Let $R\geq1$ such that $w$ is defined on $\left[  R,\infty
\right)  $ and $w(r)\leq C_{d}r^{-d}$ on $\left[  R,\infty\right)  .$ Defining
$\bar{r}$ as above and dividing (\ref{plot}) by $f_{R}$ and observing that
$f_{R}(r)\geq1,$ and $R^{k}\leq$ $R^{k-1-d}\leq r^{k-1-d},$ we deduce
\[
r^{k}\left\vert w^{\prime}(r)\right\vert \leq R^{k}\left\vert w^{\prime
}(R)\right\vert +C_{d}KR^{k-1-d}+\tilde{C}_{d}r^{k-1-d}\leq(\left\vert
w^{\prime}(R)\right\vert +C_{d}K+\tilde{C}_{d})r^{k-1-d}%
\]
on $\left[  R,\bar{r}\right)  $ and then on $\left[  R,\infty\right)  ,$ and
we conclude again.
\end{proof}

\begin{proposition}
\label{uti} (i) For any $\gamma\geq0$ if $p>2,$ any $\gamma\in\left[
0,\delta\right)  $ if $p<2,$ any solution of (\ref{un}) satifies near
$\infty$
\begin{equation}
w(r)=O(r^{-\gamma})+O(r^{-\alpha}). \label{esu}%
\end{equation}
(ii) The solution $w=w(.,a)$ of problem (\ref{un}), (\ref{ini}) satisfies%
\begin{equation}
\left\vert w(r,a)\right\vert \leq C_{\gamma}(a)((1+r)^{-\gamma}+(1+r)^{-\alpha
}), \label{esa}%
\end{equation}
where $C_{\gamma}(a)$ is continuous with respect to $a$ on $\mathbb{R}.$
\end{proposition}

\begin{proof}
(i) Here we simplify the proofs of \cite{HW} and \cite{Qi}: using equation
(\ref{un}), the function $F$ defined by
\begin{equation}
F(r)=\frac{1}{2}w^{2}+r^{-1}\left\vert w^{\prime}\right\vert ^{p-2}w^{\prime
}w, \label{the}%
\end{equation}
satisfies the relation
\begin{align*}
(r^{2\alpha}F)^{\prime}  &  =r^{2\alpha-1}(\left\vert w^{\prime}\right\vert
^{p}+(2\alpha-N)r^{-1}\left\vert w^{\prime}\right\vert ^{p-2}w^{\prime
}w-\left\vert w\right\vert ^{q+1})\\
&  \leq r^{2\alpha-1}(\left\vert w^{\prime}\right\vert ^{p}+(2\alpha
-N)r^{-1}\left\vert w^{\prime}\right\vert ^{p-2}w^{\prime}w).
\end{align*}
Assume that for some $d\geq0$ and $R>0,$ $\left\vert w(r)\right\vert \leq
Cr^{-d}$ on $\left[  R,\infty\right)  .$ Then from Proposition \ref{der} there
exists other constants $C>0$ such that $\left(  r^{2\alpha}F\right)  ^{\prime
}\leq Cr^{2\alpha-1-(d+1)p}$ on $\left[  R,\infty\right)  $. Then $F(r)\leq
C(r^{-(d+1)p}+r^{-2\alpha})$ on $\left[  R,\infty\right)  $ if $(d+1)p\neq
2\alpha;$ and $r^{-1}\left\vert w^{\prime}\right\vert ^{p-1}\left\vert
w\right\vert \leq Cr^{-(d+1)p},$ thus
\[
\left\vert w(r)\right\vert \leq C(r^{-(d+1)p/2}+r^{-\alpha})
\]
on $\left[  R,\infty\right)  .\ $We know that $w$ is bounded on $\left[
R,\infty\right)  $ from Proposition \ref{pro}. Consider the sequence $\left(
d_{n}\right)  $ defined by $d_{0}=0,$ $d_{n+1}=(d_{n}+1)p/2.$ It is increasing
and tends to $\infty$ if $p\geq2$ and to $\delta$ if $p<2.$ After a finite
number of steps, we get (\ref{esu}) by changing slightly the sequence if it
takes the value $2\alpha/p-1$.\medskip\ 

\noindent(ii) We have $\left\vert w(.,a)\right\vert $ $\leq a,$ from Theorem
\ref{exi}. Assuming that for some $d\geq0,$ $\left\vert w(r,a)\right\vert \leq
C_{d}(a)(1+r)^{-d}$ on $\left[  0,\infty\right)  ,$ and $C_{d}$ is continuous,
then
\[
\left\vert w(r,a)\right\vert \leq\tilde{C}_{d}(a)((1+r)^{-(d+1)p/2}%
+(1+r)^{-\alpha})
\]
from Proposition \ref{der}, where $\tilde{C}_{d}$ is also continuous. We
deduce (\ref{esa}) as above, and $C_{\gamma}$ is continuous, since we use is a
finite number of steps. Notice in particular that $\lim_{a\rightarrow
0}C_{\gamma}(a)=0$.\medskip
\end{proof}

As a consequence we can extend a property of zeros given in \cite[Proposition
3.1]{W1} in case $p=2,$ which improves Proposition \ref{zer}:

\begin{proposition}
\label{tuc}Assume that $\alpha<N,$ or $p>2,$ or $\alpha<\eta.$ Given $A>0,$
there exists $M(A)>0$ such that if $0<\left\vert a\right\vert \leq A,$ then
the solution $w(.,a)$ of (\ref{un}), (\ref{ini}) has at most one isolated zero
outside $\left[  0,M(A)\right]  .$
\end{proposition}

\begin{proof}
From Proposition \ref{zer}, $w(.,a)$ has a finite number of isolated zeros.
Let $\rho_{1}<\rho_{2}$ be its two last zeros, where by convention $\rho_{2}=$
$\bar{r}$ if $p>2$ and the function has a compact support $\left[  0,\bar
{r}\right]  .$ From Proposition \ref{uti}, for any $\mu>0,$ there exists
$R=R(A,\mu)>0$ such that $\max_{\left\vert a\right\vert \leq A,r\geq
R}\left\vert w(r,a)\right\vert \leq\mu^{1/(q-1)}.$ Also $\max_{\left\vert
a\right\vert \leq A,r\geq0}\left\vert w(r,a)\right\vert \leq A,$ from Theorem
\ref{exi}. As in Proposition \ref{zer}, we make the substitution (\ref{cge})
for some $d>0$. If $p>2,$ we choose $d>\alpha$, and fix $\mu=(d-\alpha)/2.$
Suppose that $\rho_{1}>R.$ Then from (\ref{tog}), denoting $\mu^{\prime
}=d^{p-1}((p-1)d-N+p),$ there exists $\rho\in\left(  \rho_{1},\rho_{2}\right)
$ such that$\rho^{p}\left\vert w(\rho)\right\vert ^{2-p}\left(  d-\alpha
-\left\vert w\right\vert ^{q-1}(\rho)\right)  \leq(p-1)d^{p-1}(d-\eta).$%
\[
\mu\rho^{p}\leq\mu^{\prime}\left\vert w(\rho)\right\vert ^{p-2}\leq\mu
^{\prime}A^{p-2}%
\]
Taking $M(A)=\max(R(A,\mu),(\mu^{\prime}\mu^{-1}A^{p-2})^{1/p}),$ we find
$\rho_{1}\leq M(A)$. If $p<2$ and $\alpha<\eta,$ taking $d\in\left(
\alpha,\eta\right)  $ and the same $\mu,$ and $M(A)=R(A,\mu)$, then $\rho
_{1}\leq M(A),$ from (\ref{tog}). If $p<2$ and $\alpha<N,$ we choose
$\mu=(N-\alpha)/2$ and $M(A)=R(A,\mu)$ and get $\rho_{1}\leq M(A)$ from
(\ref{tug}) by contradiction.
\end{proof}

\section{The case $(2-p)\alpha<p$\label{S3}}

In this paragraph, we suppose that $(2-p)\alpha<p,$ or equivalently,
\begin{equation}
p>2\text{ \qquad or }\quad\text{(}p<2\text{ and }\alpha<\delta). \label{log}%
\end{equation}

\subsection{Behaviour near infinity}

\begin{proposition}
\label{alp}Assume (\ref{log}) and $q>1$. For any solution $w$ of problem
(\ref{un}), there exists $L\in\mathbb{R}$ such that $\lim_{r\rightarrow\infty
}r^{\alpha}w=L.$
\end{proposition}

\begin{proof}
From Propositions \ref{der} and \ref{uti}, $w(r)=O\left(  r^{-\alpha}\right)
$ and $w^{\prime}(r)=O(r^{-\alpha-1})$ near $\infty.$ Indeed it follows from
(\ref{esa}) by choosing any $\gamma>\alpha$ if $p>2$ and $\gamma\in\left(
\alpha,\delta\right)  $ if $p<2$.\textbf{ }Consider the function $J_{\alpha}$
defined in (\ref{jk}). Then from (\ref{jpk}), $J_{\alpha}^{\prime}$ is
integrable at infinity: indeed $r^{\alpha-2}\left\vert w^{\prime}\right\vert
^{p-1}=O(r^{(2-p)\alpha-p-1})$ and (\ref{log}) holds, and $r^{\alpha
-1}\left\vert w\right\vert ^{q-1}w=O(r^{-1-\alpha(q-1)})$. Then $J_{\alpha}$
has a limit $L$ as $r\rightarrow\infty.$ And
\[
r^{\alpha}w=J_{\alpha}(r)-r^{\alpha-1}\left\vert w^{\prime}\right\vert
^{p-2}w^{\prime}=J_{\alpha}(r)+O(r^{(2-p)\alpha-p}),
\]
thus $\lim_{r\rightarrow\infty}r^{\alpha}w(r)=L,$ and
\begin{equation}
L=J_{\alpha}(r)+%
{\displaystyle\int\nolimits_{r}^{\infty}}
J_{\alpha}^{^{\prime}}(s)ds. \label{eli}%
\end{equation}

\end{proof}

Next we look for precise estimates of fast decaying solutions. It is easy to
obtain an approximate estimate. Since $\lim_{r\rightarrow\infty}J_{\alpha
}(r)=0,$ we find $J_{\alpha}(r)=-%
{\displaystyle\int\nolimits_{r}^{\infty}}
J_{\alpha}^{\prime}(s)ds;$ thus
\begin{equation}
\left\vert w(r)\right\vert \leq r^{-1}\left\vert w^{\prime}(r)\right\vert
^{p-1}+r^{-\alpha}%
{\displaystyle\int\nolimits_{r}^{\infty}}
s^{\alpha-1}\left(  \left\vert w\right\vert ^{q}+(N+\alpha)s^{-1}\left\vert
w^{\prime}\right\vert ^{p-1}\right)  ds \label{ajo}%
\end{equation}
Consider any $d\geq\alpha,$ with $(2-p)d<p,$ such that $w(r)=O(r^{-d})$, hence
also $w^{\prime}(r)=O(r^{-d-1})$ from Proposition \ref{der}. Then
$w(r)=O(r^{-d(p-1)-p})+O(r^{-qd})$ from (\ref{ajo}). Setting $d_{0}=\alpha$
and $d_{n+1}=\min(d_{n}(p-1)+p,qd_{n}),$ the sequence $\left(  d_{n}\right)  $
is nondecreasing. it tends to $\infty$ if $p>2,$ and to $\delta$ if $p<2.$
Thus
\begin{equation}
w(r)=o(r^{-d}),\quad\text{for any }d\geq0\text{ if }p>2,\quad\text{for any
}d<\delta\text{ if }p<2. \label{rox}%
\end{equation}
Next we give better estimates, for any solution of the problem, even changing
sign or not everywhere defined.

\begin{proposition}
\label{tim} Assume (\ref{log}). Let $w$ be any solution of (\ref{un}) such
that $\lim_{r\rightarrow\infty}r^{\alpha}w(r)=0.$

\noindent(i)If $p>2,$ then $w$ has a compact support.\medskip

\noindent(ii) If $p<2,$ then $w(r)=O(r^{-\delta})$ near $\infty.$
\end{proposition}

\begin{proof}
(i) Case $p>2.$ Assume that $w$ has no compact support. We can suppose that
$w>0$ for large $r$, from Proposition \ref{zer}. We make the substitution
(\ref{cge}) for some $d>\alpha.$ Since $w(r)=o(r^{-d}),w^{\prime
}(r)=o(r^{-d-1})$ near $\infty$ we get $y_{d}(\tau)=o(1),$ $y_{d}^{\prime
}(\tau)=o(1)$ near $\infty.$ And $\psi=dy_{d}-y_{d}^{\prime}=-r^{d+1}%
w^{\prime}$ is positive for large $\tau$ from Proposition \ref{pro}. From
(\ref{phis}),
\[
y_{d}^{\prime\prime}+(\eta-2d)y_{d}^{\prime}-d(\eta-d)y_{d}+\frac{1}%
{p-1}e^{((p-2)d+p)\tau}\psi^{2-p}\left(  y_{d}^{\prime}-(d-\alpha
)y_{d}+e^{-d(q-1)\tau}\left\vert y_{d}\right\vert ^{q-1}y_{d}\right)  =0.
\]
As in Proposition \ref{zer} the maximal points $\tau$ of $y_{d}$ remain in a
bounded set, hence $y_{d}$ is monotone for large $\tau,$ hence $y_{d}^{\prime
}(\tau)\leq0,$ and $\lim_{\tau\rightarrow\infty}e^{((p-2)d+p)\tau}\psi
^{2-p}=\lim_{r\rightarrow\infty}r^{2}\left\vert w^{\prime}\right\vert
^{2-p}=\infty.$ Then
\[
(p-1)y_{d}^{\prime\prime}=e^{((p-2)d+p)\tau}\psi^{2-p}\left(  \left\vert
y_{d}^{\prime}\right\vert (1+o(1)+(d-\alpha)y_{d}(1+o(1)\right)  .
\]
Since $d-\alpha>0,$ there exists $C>0$ such that $y_{d}^{\prime\prime}\geq
Ce^{((p-2)d+p)\tau}\psi^{3-p}$ for large $\tau,$ then
\[
-\psi^{\prime}=y_{d}^{\prime\prime}+d\left\vert y_{d}^{\prime}\right\vert \geq
Ce^{((p-2)d+p)\tau}\psi^{3-p},
\]
thus $\psi^{p-2}+Ce^{((p-2)d+p)\tau}/(d+\left\vert \delta\right\vert )$ is
nonincreasing, which is impossible.\medskip

\noindent(ii) Case $p<2.$ Let us prove that $y$ is bounded near $\infty.$ If
holds if $y$ is changing sign, from Lemma \ref{com}. Next assume that for
example $y>0$ for large $\tau,$ thus also $Y.$ If $y$ is not monotone, then
$N<\delta$ and $\lim_{\tau\rightarrow\infty}y(\tau)=\ell,$ from Lemma
\ref{com}. If $y$ is monotone, and unbounded, then is nondecreasing and
tending to $\infty.$ Then $Y\leq(\delta y)^{p-1}$ from system (\ref{sys}),
which implies $Y=o(y);$ then $y-Y>0$ for large $\tau,$ thus for any
$\varepsilon>0,$ for large $\tau,$%
\begin{align*}
(y-Y)^{\prime}  &  =(\delta-\alpha)y+(N-\delta)Y-e^{-\delta(q-1)\tau
}\left\vert y\right\vert ^{q-1}y\\
&  =(\delta-\alpha)(y-Y)+(N-\alpha)Y-e^{-\delta(q-1)\tau}\left\vert
y\right\vert ^{q-1}y\geq(\delta-\alpha-\varepsilon)(y-Y)
\end{align*}
and $y\geq y-Y\geq Ce^{(\delta-\alpha-\varepsilon)\tau},$ for some $C>0,$
which contradicts (\ref{rox}). \medskip
\end{proof}

Next we complete the estimates of Proposition \ref{tim} when $p<2.$

\begin{proposition}
\label{des}Under the assumptions of Proposition \ref{tim} with $p<2,$ if $w$
has a finite number of zeros, then
\begin{equation}
\text{(i) if }p_{1}<p,\text{ }\qquad\qquad\qquad\lim_{r\rightarrow\infty
}r^{\delta}w=\pm\ell;\qquad\qquad\qquad\qquad\qquad\qquad\qquad\qquad
\qquad\label{lim}%
\end{equation}%
\begin{equation}
\text{(ii) if }p<p_{1},\qquad\qquad\qquad\lim_{r\rightarrow\infty}r^{\eta
}w=c\qquad c\in\mathbb{R},\quad c\neq0;\qquad\qquad\qquad\qquad\qquad
\label{lum}%
\end{equation}%
\begin{equation}
\text{(iii) if }p=p_{1},\qquad\qquad\text{ }\lim_{r\rightarrow\infty}r^{N}(\ln
r)^{(N+1)/2}w=\pm\varrho,\qquad\varrho=\frac{1}{N}\left(  \frac{N(N-1)}%
{2(N-\alpha)}\right)  ^{(N+1)/2}. \label{lam}%
\end{equation}

\end{proposition}

\begin{proof}
We can assume that $w>0$ for large $r..$ Then $y,Y$ are positive for large
$\tau,$ from Proposition \ref{pro}, and $y,y^{\prime}$ are bounded from
Propositions \ref{tim} and \ref{der}. If $y$ is not monotone for large $\tau,$
then $N<\delta$ from Lemma \ref{com}, that means $p_{1}<p$ from (\ref{dn}),
and $\lim_{\tau\rightarrow\infty}y(\tau)=\ell,$ which proves (\ref{lim}). So
we can assume that $y$ is monotone for large $\tau.$ Since it is bounded,
then, from Lemma \ref{com}, either $N<\delta$ and $\lim_{\tau\rightarrow
\infty}y(\tau)=\ell$ or $0,$ or $\delta\leq N$ and $\lim_{\tau\rightarrow
\infty}y(\tau)=0.$ Suppose that $\lim_{\tau\rightarrow\infty}y(\tau)=0$. Then
$y^{\prime}(\tau)\leq0$ for large $\tau.\medskip$

\noindent\textbf{(i) Case }$p_{1}<p$ $(N<\delta).$ Then $N<\delta p,$ and from
(\ref{yss}),
\begin{equation}
(p-1)y^{\prime\prime}+(\delta p-N)\left\vert y^{\prime}\right\vert
+(\delta-N)\delta y=o(\left\vert y^{\prime}\right\vert ^{3-p})+o(y^{3-p}).
\label{moi}%
\end{equation}
Thus $y$ is concave for large $\tau,$ which is a contradiction; and
(\ref{lim}) holds. $\medskip$

\noindent\textbf{(ii) Case }$p<p_{1}$ $(\delta<N).$ We observe that
\begin{equation}
-(p-1)y^{\prime\prime}+(\delta p-N)y^{\prime}+(N-\delta)\delta y\leq0
\label{sit}%
\end{equation}
for $\tau\geq\tau_{1}$ large enough, since $\alpha<\delta;$ and we can suppose
$y(\tau)\leq1$ for $\tau\geq\tau_{1}.\ $For any $\varepsilon>0,$ the function
$\tau\longmapsto\varepsilon+e^{-\mu(\tau-\tau_{1})}$ is a solution of the
corresponding equation on $\left[  \tau_{1},\infty\right)  $, where
\begin{equation}
\mu=\eta-\delta=(N-\delta)/(p-1)>0. \label{mu}%
\end{equation}
Then $y(\tau)\leq\varepsilon+e^{-\mu(\tau-\tau_{1})}$ from the maximum
principle. Then $y(\tau)\leq e^{-\mu(\tau-\tau_{1})}$ on $\left[  \tau
_{1},\infty\right)  $. That means that $w(r)=O(r^{(p-N)/(p-1)})$ near $\infty
$, hence $w^{\prime}(r)=O(r^{(1-N)/(p-1)})$ from Proposition \ref{der}. Next
we make the substitution (\ref{cge}), with $d=\eta.$ Then functions $y_{\eta}$
and $y_{\eta}^{\prime}$ are bounded, and from (\ref{phis})
\begin{equation}
(p-1)(y_{\eta}^{\prime\prime}-\eta y_{\eta}^{\prime})=e^{(p-(2-p)\eta)\tau
}\left\vert \eta y_{\eta}-y_{\eta}^{\prime}\right\vert ^{2-p}\left(  -y_{\eta
}^{\prime}+(\eta-\alpha)y_{\eta}-e^{-\eta(q-1)\tau}\left\vert y_{\eta
}\right\vert ^{q-1}y_{\eta}\right)  ; \label{pli}%
\end{equation}
hence $(e^{-\eta\tau}y_{\eta}^{\prime})^{\prime}=O(e^{(p-(3-p)\eta)\tau}).$
Since $\lim_{\tau\rightarrow\infty}e^{-\eta\tau}y_{\eta}^{\prime}(\tau)=0,$
and $\delta<\eta$ from (\ref{dn}), we find $p<(2-p)\eta<(3-p)\eta,$ then
$e^{-\eta\tau}y_{\eta}^{\prime}(\tau)=O(e^{(p-(3-p)\eta)\tau}),$ thus
$y_{\eta}^{\prime}(\tau)=O(e^{(p-(2-p)\eta)\tau}).$ Then $y_{\eta}$ has a
limit $c\geq0$ as $\tau\rightarrow\infty,$ thus
\[
\lim_{r\rightarrow\infty}r^{\eta}w=c.
\]
Suppose that $c=0.$ Then $y_{d}(\tau)=O(e^{-\gamma_{0}\tau})$, with
$\gamma_{0}=(2-p)d-p>0.$ Assuming that $y_{d}(\tau)=O(e^{-\gamma_{n}\tau})$
for some $\gamma_{n}>0,$ then $y_{d}^{\prime}(\tau)=O(e^{-\gamma_{n}\tau})$
from Proposition \ref{der}, hence $(e^{-d\tau}y_{d}^{\prime})^{\prime
}=O(e^{(p-(3-p)d-(3-p)\gamma_{n})\tau})$, and in turn $y_{d}(\tau
)=O(e^{-\gamma_{n+1}\tau})$ with $\gamma_{n+1}=(3-p)\gamma_{n}+(2-p)d-p.$ And
$\lim_{n\rightarrow\infty}\gamma_{n}=\infty,$ thus $w(r)=o(r^{-\gamma})$ for
any $\gamma>0.$ Let use make again the substitution (\ref{cge}), with now
$d>\eta.$ The new function $y_{d}$ satisfies $\lim_{\tau\rightarrow\infty
}y_{d}(\tau)=$ $\lim_{\tau\rightarrow\infty}y_{d}^{\prime}(\tau)=0$. It is
nondecreasing near $\infty$, since $\alpha\neq d:$ indeed at each point $\tau$
large enough where $y_{d}^{\prime}(\tau)=0,$ $y_{d}^{\prime\prime}(\tau)$ has
a constant sign from (\ref{phis}). Otherwise $\lim_{\tau\rightarrow\infty
}e^{(p-(2-p)d)\tau}=0,$ since $\delta<d.$ Then%
\[
(p-1)y_{d}^{\prime\prime}+\left(  2d-\eta+o(1)\right)  \left\vert
y_{d}^{\prime}\right\vert +d(d-\eta+o(1))y_{d}=0;
\]
thus $y_{d}^{\prime\prime}$ is concave for large $\tau,$ which is a
contradiction. Thus $c>0$ and (\ref{lum}) holds.$\medskip$

\noindent\textbf{(iii) Case }$p=p_{1}$ $(\delta=N).$ Then also $\delta=\eta.$
From (\ref{sys}),%
\begin{equation}
y^{\prime}-Ny=-Y^{1/(p-1)},\qquad Y^{\prime}+Y^{1/(p-1)}=\alpha y+e^{\delta
(q-1)\tau}y^{q} \label{sysa}%
\end{equation}
hence $Y^{\prime}+Y^{1/(p-1)}\geq0,$ thus by integration, $Y(\tau)\geq
C_{1}\tau^{-(p-1)/(2-p)}$ for some $C_{1}>0$ and for large $\tau.$ From
(\ref{sysa}), there exists $K_{1}>0$ such that
\[
\left(  -Ne^{-N\tau}y\right)  ^{\prime}\geq K_{1}e^{-N\tau}\tau^{-1/(2-p)}%
\geq-\frac{K_{1}}{2}\left(  e^{-N\tau}\tau^{-1/(2-p)}\right)  ^{\prime}%
\]
for large $\tau,$ which implies a lower bound
\[
y\geq(K_{1}/2N)\tau^{-1/(2-p)}.
\]
Also $Y^{\prime}+Y^{1/(p-1)}\leq(\alpha/N+o(1))Y^{1/(p-1)},$ since $y^{\prime
}<0.$ Then for any $\varepsilon>0$,
\begin{equation}
Y^{\prime}+(\frac{N-\alpha}{N}-\varepsilon)Y^{1/(p-1)}\leq0 \label{eps}%
\end{equation}
for large $\tau.$ Taking $\varepsilon$ small enough, we deduce
\begin{equation}
Y(\tau)\leq C_{1,\varepsilon}\tau^{-(p-1)/(2-p)},\quad\text{with
}C_{1,\varepsilon}^{(2-p)/(p-1)}=\frac{p-1}{2-p}(\frac{N-\alpha}%
{N}-2\varepsilon)^{-1} \label{aps}%
\end{equation}
for large $\tau.$ Then
\[
\left(  -Ne^{-N\tau}y\right)  ^{\prime}\leq NC_{1,\varepsilon}^{1/(p-1)}%
e^{-N\tau}\tau^{-1/(2-p)}\leq-C_{1,\varepsilon}^{1/(p-1)}\left(  e^{-N\tau
}\tau^{-1/(2-p)}\right)  ^{\prime}.
\]
Thus we get an upper bound%
\[
y(\tau)\leq\frac{1}{N}C_{1,\varepsilon}^{1/(p-1)}\tau^{-1/(2-p)}.
\]
Moreover from (\ref{sysa}) and (\ref{eps}), $\left\vert Y^{\prime}%
(\tau)\right\vert \leq Y^{1/(p-1)}(\tau)$ for large $\tau$, hence from
(\ref{aps}), $y^{\prime\prime}-Ny^{\prime}=-Y^{1/(p-1)}Y^{\prime}=O\left(
\tau^{-(3-p)/(2-p)}\right)  .$ Then $\left(  e^{-N\tau}y^{\prime}\right)
^{\prime}=O(e^{-N\tau}\tau^{-(3-p)/(2-p)}),$ thus $y^{\prime}=O(\tau
^{-(3-p)/(2-p)}),$ hence $y^{\prime}=o(y)$ from the lower estimate of $y.$Then
for any $\varepsilon>0,$
\[
Y^{\prime}+(\frac{N-\alpha}{N}-\varepsilon)Y^{1/(p-1)}\geq0
\]
for large $\tau;$ then
\[
Y(\tau)\geq C_{2,\varepsilon}\tau^{-(p-1)/(2-p)},\quad\text{with
}C_{2,\varepsilon}^{(2-p)/(p-1)}=\frac{p-1}{2-p}(\frac{N-\alpha}%
{N}+2\varepsilon)^{-1}%
\]
for large $\tau.$ Thus
\[
\lim_{\tau\rightarrow\infty}\tau^{-(p-1)/(2-p)}Y(\tau)=(\frac{p-1}{2-p}%
\frac{N}{N-\alpha})^{(p-1)/(2-p)}=\lim_{\tau\rightarrow\infty}(\tau
^{-1/(2-p)}Ny(\tau))^{p-1},
\]
so that $\lim_{\tau\rightarrow\infty}(\tau^{-1/(2-p)}y(\tau))=\varrho$ and
(\ref{lam}) holds.\medskip
\end{proof}

We can get an asymptotic expansion of the slow decaying solutions, which in
fact covers the case $p=2,$ where we find again the results of \cite[Theorem
1]{W1}.

\begin{proposition}
\label{expa}Assume (\ref{log}). Let $w$ be any solution of (\ref{un}), such
that $L=\lim_{r\rightarrow\infty}r^{\alpha}w>0.$ Then
\begin{equation}
\lim_{r\rightarrow\infty}r^{\alpha+1}w^{\prime}=-\alpha L, \label{imi}%
\end{equation}
and
\begin{equation}
w(r)=\left\{
\begin{array}
[c]{c}%
r^{-\alpha}\left(  L+\left(  K+o(1)\right)  r^{-k}\right)  ,\qquad\qquad
\qquad\text{if }\left(  q+1-p\right)  \alpha>p,\\
r^{-\alpha}\left(  L+\left(  K+M+o(1)\right)  r^{-\alpha(q-1)}\right)
,\qquad\text{if }\left(  q+1-p\right)  \alpha=p,\\
r^{-\alpha}\left(  L+\left(  M+o(1)\right)  r^{-\alpha(q-1)}\right)
,\qquad\qquad\text{if }\left(  q+1-p\right)  \alpha<p,
\end{array}
\right.  \label{exw}%
\end{equation}
where
\[
k=p-(2-p)\alpha,\qquad K=\frac{\left(  \alpha(p-1)-(N-p)\right)  \left(
\alpha L\right)  ^{1/(p-1)}}{k},\qquad M=\frac{L^{q}}{\alpha(q-1)}.
\]
Moreover differentiating term to term gives an expansion of $w^{\prime}.$
\end{proposition}

\begin{proof}
We make the substitution (\ref{cge}) with $d=\alpha,$ thus $w(r)=r^{-\alpha
}y_{\alpha}(\tau).$ For large $r,$ $w^{\prime}(r)=r^{-(\alpha+1)}(\alpha
y_{\alpha}(\tau)-y_{\alpha}^{\prime}(\tau))<0,$ thus $\alpha y_{\alpha
}-y_{\alpha}^{\prime}>0$ for large $\tau.$ And (\ref{sysd}) becomes:%
\begin{equation}
\left\{
\begin{array}
[c]{c}%
y_{\alpha}^{\prime}=\alpha y_{\alpha}-Y_{\alpha}^{1/(p-1)}\qquad\qquad
\qquad\qquad\qquad\qquad\qquad\qquad\quad\\
Y_{\alpha}^{\prime}=(p-1)(\alpha-\eta)Y_{\alpha}+e^{k\tau}(\alpha y_{\alpha
}-Y_{\alpha}^{1/(p-1)}+e^{-\alpha(q-1)\tau}y_{\alpha}^{q}).
\end{array}
\right.  \label{syph}%
\end{equation}
The function $y_{\alpha}$ converges to $L,$ and $y_{\alpha}^{\prime}$ is
bounded near $\infty,$ since $w^{\prime}=O(r^{-\left(  \alpha+1\right)  })$
near $\infty,$ thus $Y_{\alpha}$ is bounded. Either $Y_{\alpha}$ is monotone
for large $\tau$, then it has a finite limit $\lambda;$ then $y_{\alpha
}^{\prime}$ converges to $\alpha L-\lambda^{1/(p-1)};$ thus $\lambda=\left(
\alpha L\right)  ^{1/(p-1)}.$ Or for large $\tau,$ the extremal points of
$Y_{\alpha}$ form an increasing sequence $\left(  \tau_{n}\right)  $ tending
to $\infty.$ Then
\[
Y_{\alpha}\left(  \tau_{n}\right)  ^{1/(p-1)}=\alpha y_{\alpha}\left(
\tau_{n}\right)  +e^{-\alpha(q-1)\tau_{n}}y_{\alpha}^{q}(\tau_{n}%
)+(p-1)(\alpha-\eta)e^{-k\tau_{n}}Y_{\alpha}(\tau_{n})
\]
thus $\lim Y_{\alpha}\left(  \tau_{n}\right)  =\left(  \alpha L\right)
^{1/(p-1)}.$ In any case $\lim_{\tau\rightarrow\infty}Y_{\alpha}\left(
\tau\right)  =\left(  \alpha L\right)  ^{1/(p-1)},$ which is equivalent to
(\ref{imi}), and implies $\lim_{\tau\rightarrow\infty}y_{\alpha}^{\prime
}\left(  \tau\right)  =0.$ Now consider $Y_{\alpha}^{\prime}.$ Either it is
monotone for large $\tau,$ thus $\lim_{\tau\rightarrow\infty}Y_{\alpha
}^{\prime}\left(  \tau\right)  =0;$ or for large $\tau,$ the extremal points
of $Y_{\alpha}^{\prime}$ form an increasing sequence $\left(  s_{n}\right)  $
tending to $\infty.$ Then $Y_{\alpha}^{\prime\prime}\left(  \tau_{n}\right)
=0,$ then by computation, at point $\tau=s_{n},$%
\begin{align*}
\left(  \frac{1}{p-1}Y_{\alpha}^{(2-p)/(p-1)}-(p-1)(\alpha-\eta)e^{-k\tau
}\right)  Y_{\alpha}^{\prime}  &  =\left(  p+\alpha(p-1)+qe^{-\alpha(q-1)\tau
}y_{\alpha}^{q-1}\right)  y_{\alpha}^{\prime}\\
&  +(k-\alpha(q-1))e^{-\alpha(q-1)\tau}y_{\alpha}^{q}%
\end{align*}
thus $\lim Y_{\alpha}^{\prime}\left(  s_{n}\right)  =0.$ In any case,
$\lim_{\tau\rightarrow\infty}Y_{\alpha}^{\prime}\left(  \tau\right)  =0.$ From
(\ref{syph}), we deduce
\[
y_{\alpha}^{\prime}=-e^{-\alpha(q-1)\tau}y_{\alpha}^{q}-e^{-k\tau
}((p-1)(\alpha-\eta)Y_{\alpha}-Y_{\alpha}^{\prime})=-(L^{q}+o(1))e^{-\alpha
(q-1)\tau}-k(K+o(1))e^{-k\tau}%
\]
thus $y_{\alpha}^{\prime}=-k(K+o(1))e^{-k\tau}$ if $\alpha(q-1)>k,$ or
equivalently $\left(  q+1-p\right)  \alpha>p;$ and $y_{\alpha}^{\prime
}=-(kK+L^{q}+o(1))e^{-k\tau}$ if $\alpha(q-1)=k;$ and $y_{\alpha}^{\prime}=$
$-(L^{q}+o(1))e^{-\alpha(q-1)\tau}$ if $\alpha(q-1)<k.$ The estimates
(\ref{exw}) follow by integration.This gives also an expansion of the
derivatives, by computing $w^{\prime}=-r^{-(\alpha+1)}(\alpha y_{\alpha
}-y_{\alpha}^{\prime}):$%
\[
w^{\prime}(r)=\left\{
\begin{array}
[c]{c}%
-r^{-(\alpha+1)}\left(  \alpha L+(\alpha+k)\left(  K+o(1)\right)
r^{-k}\right)  ,\qquad\quad\text{if }\left(  q+1-p\right)  \alpha>p,\\
-r^{-(\alpha+1)}\left(  \alpha L+(\alpha+k)\left(  K+M+o(1)\right)
r^{-k}\right)  ,\qquad\text{if }\left(  q+1-p\right)  \alpha=p,\\
-r^{-(\alpha+1)}\left(  \alpha L+\alpha q(M+o(1))r^{-\alpha(q-1)}\right)
,\qquad\quad\text{if }\left(  q+1-p\right)  \alpha<p;
\end{array}
\right.
\]
which corresponds to a derivation term to term.
\end{proof}

\subsection{Continuous dependence and sign properties}

Next we extend an important property of continuity with respect to the initial
data, given in \cite{HW} in the case $p=2.$ The proof is different; it follows
from the estimates of Proposition (\ref{der}) and from the expression of
$L(a)$ in terms of function $J_{\alpha}.$

\begin{theorem}
\label{cont}Assume (\ref{log}). For any solution $w=w(.,a)$ of problem
(\ref{un}), (\ref{ini}), setting $L=L(a),$ the function $a\longmapsto L(a)$ is
continuous on whole $\mathbb{R}.$ Moreover the family of functions $\left(
a\longmapsto(1+r)^{\alpha}w(r,a)\right)  _{r\geq0}$ is equicontinuous on
$\mathbb{R}.$\medskip
\end{theorem}

\begin{proof}
\textbf{Let} $a_{0}\in\mathbb{R}.$ From Propositions \ref{der} and
(\ref{uti}), there exists a neighborhood $V$ of $a_{0}$ and a constant
$C=C(V)>0$ such that for any $a\in V,$%
\begin{equation}
\left\vert w(r,a)\right\vert \leq C(1+r)^{-\alpha},\quad\quad\left\vert
w^{\prime}(r,a)\right\vert \leq C(1+r)^{-(\alpha+1)}, \label{uni}%
\end{equation}
From (\ref{eli}), we have for any $r\geq1,$
\begin{equation}
L(a)=J_{\alpha}(r,a)+%
{\displaystyle\int\nolimits_{r}^{\infty}}
J_{\alpha}^{^{\prime}}(s,a)ds=%
{\displaystyle\int\nolimits_{0}^{\infty}}
J_{\alpha}^{^{\prime}}(s,a)ds \label{ele}%
\end{equation}
where $J_{\alpha}(r,a)=r^{\alpha}\left(  w(r,a)+r^{-1}\left\vert w^{\prime
}\right\vert ^{p-2}w^{\prime}(r,a)\right)  ,$ since $J_{\alpha}(0,a)=0.$ Then
with a new constant $C=C(V),$ for any $a\in V,$
\[%
{\displaystyle\int\nolimits_{r}^{\infty}}
\left\vert J_{\alpha}^{^{\prime}}(s,a)\right\vert ds\leq C\left(
r^{-\alpha(q-1)}+r^{-(p-\alpha(2-p)}\right)  ;
\]
hence for any $\varepsilon>0,$ there exists $r_{\varepsilon}\geq1$ such that
\[
\sup_{a\in V}%
{\displaystyle\int\nolimits_{r_{\varepsilon}}^{\infty}}
\left\vert J_{\alpha}^{^{\prime}}(s,a)\right\vert ds\leq\varepsilon.
\]
From Remark \ref{dep}, $w(.,a)$ depends continuously on $a$ on any compact
set, thus also $J_{\alpha}^{^{\prime}}(.,a).$ Then there exists a neighborhood
$V_{\varepsilon}$ of $a_{0}$ contained in $V$ such that
\[
\sup_{a\in V\varepsilon}%
{\displaystyle\int\nolimits_{0}^{r_{\varepsilon}}}
\left\vert J_{\alpha}^{^{\prime}}(r_{\varepsilon},a)-J_{\alpha}^{^{\prime}%
}(r_{\varepsilon},a_{0})\right\vert \leq\varepsilon,
\]
and consequently $\left\vert L(a)-L(a_{0})\right\vert \leq3\varepsilon.$ This
proves that $L$ is continuous at $a_{0}.$ Moreover
\[
\sup_{a\in V\varepsilon}\sup_{r\in\left[  0,\infty\right)  }\left\vert
J_{\alpha}(r,a)-J_{\alpha}(r,a_{0})\right\vert \leq2\varepsilon,
\]
thus the family of functions $\left(  a\longmapsto J_{\alpha}(r,a)\right)
_{r\geq0}$ is equicontinuous at $a_{0}.$ Next for any $r\geq1$ and any $a\in
V,$
\[
\left\vert r^{\alpha}w(r,a)-J_{\alpha}(r,a)\right\vert =r^{\alpha-1}\left\vert
w^{\prime}(r,a)\right\vert ^{p-1}\leq Cr^{(2-p)\alpha-p},
\]
thus for any $\varepsilon>0,$ there exists $\tilde{r}_{\varepsilon}\geq
r_{\varepsilon}$ such that
\[
\sup_{a\in V,r\geq\tilde{r}_{\varepsilon}}\left\vert r^{\alpha}%
w(r,a)-J_{\alpha}(r,a)\right\vert \leq\varepsilon.
\]
It implies
\[
\sup_{a\in V_{\varepsilon},r\geq\tilde{r}_{\varepsilon}}\left\vert
(1+r)^{\alpha}(w(r,a)-w(r,a_{0}))\right\vert \leq(2^{\alpha}+2)\varepsilon.
\]
And there exists a neighborhood $\tilde{V}_{\varepsilon}$ of $a_{0}$ contained
in $V_{\varepsilon},$ such that
\[
\sup_{a\in\tilde{V}_{\varepsilon},r\leq\tilde{r}_{\varepsilon}}\left\vert
(1+r)^{\alpha}(w(r,a)-w(r,a_{0}))\right\vert \leq\varepsilon.
\]
Then
\[
\sup_{a\in\tilde{V}_{\varepsilon},r\in\left[  0,\infty\right)  }\left\vert
(1+r)^{\alpha}(w(r,a)-w(r,a_{0}))\right\vert \leq(2^{\alpha}+2)\varepsilon,
\]
which shows that the family of functions $a\longmapsto(1+r)^{\alpha}w(r,a)$
$(r\geq0)$ is equicontinuous at $a_{0}.\medskip$
\end{proof}

As a consequence we obtain some results concerning the number of zeros of the solutions

\begin{theorem}
\label{pn}Assume (\ref{log}).

\noindent(i) Suppose that for some $a_{0}>0,$ $w(.,a_{0})$ has a finite number
of isolated zeros, denoted by $N(a_{0}).$ If $L(a_{0})\neq0,$ then
$N(a)=N(a_{0})$ for any $a$ close to $a_{0}.$

\noindent(ii) Suppose $q<q^{\ast}.$ Then $\left\{  a>0:L(a)=0\right\}  $ is
unbounded from above. Moreover there exists a increasing sequence $\left(
a_{m}\right)  $ tending to $\infty,$ such that $w(.,a_{m})$ has at least $m+1$
isolated zeros and $L(a_{m})=0.$

\noindent(iii) Suppose $q<q^{\ast},$ $p<2$ and $\alpha<N$. Then for any
$m\in\mathbb{N},$
\[
\bar{a}_{m}=\inf\left\{  a>0:N(a)\geq m+1\right\}  \in\left(  0,\infty\right)
,
\]
and if $m\geq1,$ then $w(.,\bar{a}_{m})$ has precisely $m$ zeros and
$L(\bar{a}_{m})=0.$
\end{theorem}

\begin{proof}
(i) Let $r_{1}<r_{2}<..<r_{N(a_{0})}$ be the isolated zeros of $w(.,a_{0}).$
Since $L(a_{0})\neq0,$ there are no other zeros, and there exists
$\varepsilon>0$ such that $\inf_{r\geq r_{N(a_{0})}+1}r^{\alpha}\left\vert
w(r,a_{0})\right\vert \geq\varepsilon$. From Theorem \ref{cont}, there exists
a neighborhood $V_{\varepsilon}$ of $a_{0}$ such that $\inf_{r\geq
r_{N(a_{0})}+1}r^{\alpha}\left\vert w(r,a)\right\vert \geq\varepsilon/2$ for
any $a\in V_{\varepsilon}.$ From Remark \ref{dep}, there exists a neighborhood
$\tilde{V}_{\varepsilon}\subset V_{\varepsilon}$ such that $w(r,a)$ has
exactly $N(a_{0})$ zeros on $\left[  0,r_{N(a_{0})}+1\right]  ,$ hence
$N(a)=N(a_{0}).\medskip$

\noindent(ii) Assume that for some $a^{\ast}>0,$ $L(a)\neq0$ for any
$a\in\left(  a^{\ast},\infty\right)  .$ From Proposition \ref{zer}, (iii) and
(iv), $w(.,a)$ has a finite number of isolated zeros $N(a)$. The set
\[
\left\{  a\in\left(  a^{\ast},\infty\right)  :N(a)=N(a^{\ast})+1\right\}
\]
is closed in $\left(  a^{\ast},\infty\right)  $ since $N$ is locally constant,
and open; then $N(a)$ is constant on $\left(  a^{\ast},\infty\right)  ,$ which
contradicts Proposition \ref{sig}. Moreover there exists a increasing sequence
$\left(  a_{m}^{\star}\right)  $ tending to $\infty$ such that $w(.,a_{m}%
^{\star})$ has at least $m+1$ isolated zeros; as above it cannot happen that
$L(a)\neq0$ for any $a\in\left(  a_{m}^{\ast},\infty\right)  ,$ hence there
exists $a_{m}\geq a_{m}^{\ast},$ such that $w(.,a_{m})$ has at least $m+1$
isolated zeros and $L(a_{m})=0.\medskip$

\noindent(iii) Here $w(.,a)$ has only isolated zeros. Following the proof of
\cite[Propositions 3.5 and 3.7]{W1}, for any $m\in\mathbb{N},$ the set
$B_{m}=\left\{  a>0:\text{ }N(a)\geq m+1\right\}  $ is open and $z_{m}%
(a)=m^{\text{th}}$ zero of $w(.,a)$ depends continuously on $a$. Using
Proposition \ref{tuc}, one can show that, for any $a_{0}>0,$ $N(a)=N(a_{0})$
or $N(a_{0})+1$ for any $a$ in some neighborhood of $a_{0}.$ Then necessarily
$\bar{a}_{m}\not \in B_{m},$ and $N(\bar{a}_{m})=m,$ and $L(\bar{a}_{m})=0$ by
contradiction in (i).\medskip
\end{proof}

\begin{remark}
When $q<q^{\ast}$ and $p>2,$ for any $a_{0}>0,$ we have $N(a)\geq N(a_{0})$
for any $a$ in some neighborhood of $a_{0}$, but we cannot prove that
$N(a)\leq$ $N(a_{0})+2,$ thus we have no specific information of the number of
zeros of the compact support solutions.
\end{remark}

\subsection{Existence of nonnegative solutions}

Here we study the existence of nonegative solutions of equation (\ref{un}). If
such solutions exist, then either $p_{1}<p$ and $\alpha<N,$ from From
Proposition \ref{zer}, or $p<p_{1},$ thus $\alpha<\delta\leq N;$ in any case
$\alpha<N.$ Reciprocally, when $\alpha<N,$ we first prove the existence of
slow decaying solutions for $\left\vert a\right\vert $ small enough.

\begin{proposition}
\label{ora}Assume (\ref{log}), and $\alpha<N.$ Let $\underline{a}>0$ be
defined at Proposition \ref{zer}. Then for any $a\in\left(  0,\underline
{a}\right]  $, $w(r,a)>0$ on $\left[  0,\infty\right)  ,$ and $L(a)>0.$
\end{proposition}

\begin{proof}
Let $a\in\left(  0,\underline{a}\right]  .$ By construction of $\underline
{a},$ $w=w(r,a)>0,$ from Proposition \ref{zer}, and function $J_{N}$ is
nondecreasing and $J_{N}(0)=0;$ and $J_{N}(r)\leq r^{N}w$ near $\infty,$ from
Proposition \ref{pro}. Assume that $L(a)=0.$ Then $p<2$ from Proposition
\ref{tim}. From Proposition \ref{des}, either $N<\delta,$ and $r^{N}%
w=O(r^{N-\delta})$; or $\delta<N$ and $N<\eta$ from (\ref{dn}), and
$r^{N}w=O(r^{N-\eta})$; or $\delta=N$ and $r^{N}w=O(\ln r)^{-(N+1)/2}$. In any
case, $\lim\sup_{r\rightarrow\infty}J_{N}(r)=0$; then $J_{N}\equiv0,$ thus
$J_{N}^{\prime}\equiv0,$ which is impossible.\medskip
\end{proof}

Next we consider the subcritical case $1<q<q^{\ast}$ and prove the existence
of fast decaying solutions. Notice that in that range $p>p_{2};$ if moreover
$1<q<q_{1},$ then $p>p_{1}.$

\begin{theorem}
\label{fast} Assume (\ref{log}) and $\alpha<N,$ and $1<q<q^{\ast}.$ Then there
exists $a>0$ such that $w(.,a)$ is nonnegative and such that $L(a)=0.$ If
$p>2,$ it has a compact support. If $p<2,$ it is positive and satisfies
(\ref{lim}), (\ref{lum}) or (\ref{lam}).
\end{theorem}

\begin{proof}
Let
\begin{equation}
A=\left\{  a>0:w(.,a)>0\;\text{on }\left(  0,\infty\right)  \;\text{and
}L(a)>0\right\}  ,\text{ } \label{ana}%
\end{equation}%
\begin{equation}
B=\left\{  a>0:w(.,a)\quad\text{has at least an isolated zero}\right\}  .
\label{bnb}%
\end{equation}
From Proposition \ref{ora} and \ref{sig}, $A$ and $B$ are nonempty:
$A\supset\left(  0,\underline{a}\right]  $ and $B\supset\left[  \overline
{a},\infty\right)  .$ From the local continuous dependence of the solutions on
the initial value, $B$ is open. For any $a_{0}\in A,$there exists
$\varepsilon>0$ such that $\min_{r\geq0}(1+r)^{\alpha}w(r,a_{0})\geq
\varepsilon.$ From Theorem \ref{cont}, there exists a neighborhood
$V_{\varepsilon}$ of $a_{0}$ such that $\min_{r\geq0}(1+r)^{\alpha}%
w(r,a)\geq\varepsilon/2$ for any $a\in V_{\varepsilon},$ hence $V_{\varepsilon
}\subset A,$ thus $A$ is open. Let $a_{\inf}=\inf B>\underline{a}$ and
$a_{\sup}=\sup A<\overline{a}.$ Taking $a=a_{\inf}$ or $a_{\sup},$ then
$w(.,a)$ is nonnegative, positive if $p<2,$ and $L(a)=0,$ and the conclusion
follows from Proposition \ref{des}. We cannot assert that $a_{\inf}=a_{\sup}$.
\end{proof}

\begin{remark}
As it was noticed in \cite{SW} for $p=2,$ there exists an infinity of pairs
$a_{1},a_{2}$ such that $0<a_{1}<a_{2}<a_{\inf},$ thus $w(.,a_{1})>0,$
$w(.,a_{2})>0,$ and $L(a_{1})=L(a_{2}).$ Indeed from the continuity of $L$
proved at Theorem \ref{cont}, $L$ attains at least twice any value in $\left(
0,\max_{\left[  0,a_{\inf}\right]  }L\right)  .$
\end{remark}

In the supercritical case $q\geq q^{\ast}$ we give sufficient conditions
assuring that all the solutions are positive, and then slowly decaying. Recall
that $q^{\ast}\leq1$ whenever $p\leq p_{2}.$

\begin{theorem}
\label{pos} Assume (\ref{log}) and one of the following conditions:

\noindent(i) $p_{2}<p$ and $\alpha\leq N/2$ and $q\geq q^{\ast};$

\noindent(ii) $p\leq p_{2}$ and $1<q.$

\noindent(iii) $p_{2}<p$ and $N/2<\alpha<(N-1)p^{\prime}/2$ and $q\geq
q_{\alpha}^{\ast},$ where $q_{\alpha}^{\ast}>q^{\ast}$ is given by
\begin{equation}
\frac{1}{q_{\alpha}^{\ast}+1}=\frac{N-1}{2\alpha}-\frac{1}{p^{\prime}}.
\label{etol}%
\end{equation}
Then for any $a>0,$ $w(r,a)>0$ on $\left[  0,\infty\right)  $, and $L(a)>0$.
\end{theorem}

\begin{proof}
We use the function $V=V_{\lambda,\sigma,e}$ defined at (\ref{vla}) , where
$\lambda>0,\sigma,e$ will be chosen after. It is continuous at $0$ and
$V_{\lambda,\sigma,e}(0)=0,$ from (\ref{pri}). Suppose that $w(r_{0})=0$ for
some first real $r_{0}>0.$ Then $V_{\lambda,\sigma,e}(r_{0})=r_{0}%
^{N}\left\vert w^{\prime}(r_{0})\right\vert ^{p}/p^{\prime}\geq0.$ Suppose
that for some $\lambda,\sigma,e$, the five terms giving $V^{\prime}$ are
nonpositive. Then $V\equiv V^{\prime}\equiv0$ on $\left[  0,r_{0}\right]  ,$
hence $rw^{\prime}+(\sigma-e+\alpha)w/2\equiv0,$ $r^{(\sigma-e+\alpha)/2}w$ is
constant, hence $w\equiv0$ if $\sigma-e+\alpha\neq0,$ or $w\equiv a$ if
$\sigma-e+\alpha=0.$ It is impossible since $w(0)\neq w(r_{0}).\medskip$

\noindent\textbf{Case (i)}. We take $\lambda=N$ and $\sigma=(N-p)/p$ and
$e=\sigma+\alpha-N,$ thus%
\begin{equation}
V(r)=r^{N}\left(  \frac{\left\vert w^{\prime}\right\vert ^{p}}{p^{\prime}%
}+\frac{\left\vert w\right\vert ^{q+1}}{q+1}+(\frac{N-p}{p}+\alpha
-N)\frac{w^{2}}{2}+\frac{N-p}{p}r^{-1}w\left\vert w^{\prime}\right\vert
^{p-2}w^{\prime}\right)  , \label{waa}%
\end{equation}%
\begin{equation}
r^{1-N}V^{\prime}(r)=-\left(  \frac{N-p}{p}-\frac{N}{q+1}\right)  \left\vert
w\right\vert ^{q+1}-\frac{N+2}{4p}(p-p_{2})\left(  N-2\alpha\right)
w^{2}-\left(  rw^{\prime}+\frac{N}{2}w\right)  ^{2} \label{wab}%
\end{equation}
and all the terms are nonpositive from our assumptions, thus $w>0$ on $\left[
0,\infty\right)  .$ Moreover suppose that $L(a)=0.$ Then $p<2,$ and from
Proposition \ref{tim}, $V(r)=O(r^{N-2\delta})$ as $r\rightarrow\infty,$ thus
$\lim_{r\rightarrow\infty}V(r)=0,$ since $N<2\delta$ from (\ref{ddn}). Then
$V\equiv0$ on $\left[  0,\infty\right)  $ which is a contradiction.\medskip

\noindent\textbf{Case (ii)}. We take $\lambda=N=2\sigma$ and $e=\alpha-N/2,$
thus
\begin{equation}
r^{1-N}V^{\prime}(r)=-\frac{N+2}{2p}(p_{2}-p)\left\vert w^{\prime}\right\vert
^{p}-\frac{N(q-1)}{2q+1}\left\vert w\right\vert ^{q+1}-\left(  rw^{\prime
}+Nw\right)  ^{2}, \label{wac}%
\end{equation}
and all the terms are nonpositive, and again $w>0$ on $\left[  0,\infty
\right)  .$ If $L(a)=0,$ we find $V(r)=O(r^{N-\eta})$ near $\infty,$ from
Proposition \ref{tim}, since $p\leq p_{2}<p_{1,}.$ Then $\lim_{r\rightarrow
\infty}V(r)=0,$ hence again a contradiction.\medskip

\noindent\textbf{Case (iii)}.\textbf{ }We take\textbf{ }$\lambda=2\alpha$ and
$\sigma=N-1-2\alpha/p^{\prime}$ and $e=\sigma-\alpha,$ thus
\[
r^{1-2\alpha}V^{\prime}(r)=-\left(  \sigma-\frac{2\alpha}{q+1}\right)
\left\vert w\right\vert ^{q+1}+\sigma(2\alpha-N)r^{-1}w\left\vert w^{\prime
}\right\vert ^{p-2}w^{\prime}-\left(  rw^{\prime}+\alpha w\right)  ^{2}.
\]
Here the first term is nonpositive from (\ref{etol}), and also the second
term, since $\sigma>0,$ $N/2\leq\alpha$ and $w^{\prime}<0$ on $\left(
0,r_{0}\right)  ,$ from Proposition \ref{pro}, hence again $w>0$ on $\left[
0,\infty\right)  .$ If $L(a)=0,$ then $p<2.$ From Proposition \ref{tim},
either $p_{1}<p$ and $V(r)=O(r^{2(\alpha-\delta)})$ near $\infty,$ where
$\alpha<\delta;$ or $p<p_{1}$ and $V(r)=O(r^{2(\alpha-\eta)})$, and
$\alpha<\delta<\eta$ from (\ref{dn}); or $p=p_{1}$ and $V(r)=O(\ln
r^{-(N+1)/2}).$ In any case $\lim_{r\rightarrow\infty}V(r)=0,$ hence again a contradiction.
\end{proof}

\begin{remark}
With no hypothesis on $p,$ if $w(r_{0})=0$ for some real $r_{0},$ then from
(\ref{waa}), (\ref{wab}),
\[
\left(  \frac{N-p}{p}-\frac{N}{q+1}\right)
{\displaystyle\int\nolimits_{0}^{r_{0}}}
r^{N-1}\left\vert w\right\vert ^{q+1}dr+\frac{(N+2)p-2N}{4p}\left(
N-2\alpha\right)
{\displaystyle\int\nolimits_{0}^{r_{0}}}
r^{N-1}w^{2}dr
\]%
\[
+%
{\displaystyle\int\nolimits_{0}^{r_{0}}}
r^{N-1}\left(  rw^{\prime}+\frac{N}{2}w\right)  ^{2}dr=0
\]
As in \cite{PTW} such a relation can be extended to the nonradial case and
then applied to nonradial solutions $w.$
\end{remark}

\begin{remark}
Property (ii) was proved for equation (\ref{hog}) in \cite{QW}. It is new in
the general case. It can be also obtained by using the energy function $W$
defined at (\ref{wt}) instead of $V.$

The result (iii) is new. Is also true when $p=2:$ if $N/2<\alpha<N-1$ and
$q\geq q_{\alpha}^{\ast},$ where $q_{\alpha}^{\ast}=(3\alpha-N+1)/(N-1-\alpha
)>q^{\ast}$, we prove that all the solutions are ground states, with a slow
decay$.$In the case $p=2,$ $q=q^{\ast}$ it had been shown by variational
methods in \cite{EK} that there exist ground states with a fast decay$,$
whenever $N/2<\alpha<N$ when $N\geq4,$ or if $2<\alpha<3$ when $N=3$; moreover
from \cite{AP}, they do not exist when $1<\alpha\leq2.$ Apparently nothing was
known beyond the critical case.
\end{remark}

\begin{remark}
If $1<p\leq p_{1},$ then the condition $\alpha<(N-1)p^{\prime}/2$ is always
satisfied, since $\alpha<\delta\leq N\leq(N-1)p^{\prime}/2.$ If $p_{1}<p,$ our
conditions imply $\alpha<N,$ which was a necessary condition in order to get
positive solutions, from Proposition \ref{zer}.
\end{remark}

\subsection{Oscillation or nonoscillation criteria}

Our next result concerns the case $p<2,$ and $N\leq\alpha,$ thus $N\leq
\alpha<\delta$ from (\ref{log}), where there exists no positive solutions: all
the solutions are changing sign. It is new, and uses the ideas of \cite{Bi1}
for the problem without source (\ref{hog}). It involves the coefficient
$\alpha^{\ast}$ defined at (\ref{eto}), which here satisfies $\alpha^{\ast
}<\delta,$ and the energy function $W$ defined at (\ref{wtp}); we use the
notations $\mathcal{W},\mathcal{U},\mathcal{H},\mathcal{L},\mathcal{S}$ of
Section \ref{S21}.

\begin{theorem}
\label{osci}Assume (\ref{log}), $p<2,$ and $N\leq\alpha$.

\noindent(i) If $\alpha<\alpha^{\ast},$ then any solution $w(.,a)$ $(a\neq0)$
has a finite number of zeros.

\noindent(ii) There exists $\overline{\alpha}\in\left(  \max(N,\alpha^{\ast
}),\delta\right)  $ such that for any $\alpha\in\left(  \overline{\alpha
},\delta\right)  $, any solution $w(.,a)$ has a infinity of zeros.
\end{theorem}

\begin{proof}
(i) Suppose $N\leq\alpha<\alpha^{\ast}$ (which implies $p>3/2).$ In the phase
plane $(y,Y)$ of system (\ref{sys}), the stationary point $M_{\ell}$ is in the
domain $\mathcal{S}$ of boundary $\mathcal{L}.$ Indeed denote $P_{\mu}%
=(\mu,(\delta\mu)^{p-1})$ for any $\mu>0.$ Setting $\lambda=\delta
^{-1}((2\delta-N)(p-1))^{1/(2-p)},$ the point $P_{\lambda}$ is on the curve
$\mathcal{L}.$ Then $(\theta\lambda,(\theta\delta\lambda)^{p-1})\in
\mathcal{S}$ for any $\theta\in\left[  0,1\right)  $, and $\alpha<\alpha
^{\ast}\Leftrightarrow\ell<\lambda,$ thus $P_{\ell}=M_{\ell}\in\mathcal{S},$
and there exists $\varepsilon\in\left(  0,1\right]  $ such that $P_{\ell
+\varepsilon}\in\mathcal{S}.$ Now for any $\mu>0$ such that $P_{\mu}%
\in\mathcal{S},$ the square $\mathcal{K}_{\mu}=\left\{  (y,Y)\in\mathbb{R}%
^{2}:\left\vert y\right\vert \leq\mu,\left\vert Y\right\vert \leq(\delta
\mu)^{p-1}\right\}  $ is contained in $\mathcal{S}.$ Indeed $\mathcal{H}%
(\mu,(\delta\mu)^{p-1})=(\delta\mu)^{2-p}/(p-1),$ and for any $\xi,\zeta
\in\left[  -1,1\right]  $
\[
\mathcal{H}(\xi\mu,\zeta(\delta\mu)^{p-1})=(\delta\mu)^{2-p}\frac
{\xi-\left\vert \zeta\right\vert ^{(2-p)/(p-1)}}{\left\vert \xi\right\vert
^{(2-p)/(p-1)}-\zeta}\leq\mathcal{H}(\mu,(\delta\mu)^{p-1}),
\]
since the quotient is majorized by $1/(p-1)$ if $\xi\zeta>0,$ and by $1$ if
$\xi\zeta<0,$ because $p>3/2.$ From Lemma \ref{com},iv, $(y\left(
\tau\right)  ,Y\left(  \tau\right)  )\in\mathcal{K}_{\ell+\varepsilon}$ for
$\tau\geq\tau\left(  \varepsilon\right)  $ large enough, thus $(y\left(
\tau\right)  ,Y\left(  \tau\right)  )\in\mathcal{S}$. Thus $\mathcal{U}%
(y\left(  \tau\right)  ,Y\left(  \tau\right)  )\geq0.$ Consider the function
\begin{equation}
\tau\mapsto\Psi(\tau)=W\left(  \tau\right)  -\frac{\delta(q-1)}{q+1}%
{\displaystyle\int\limits_{\tau}^{\infty}}
e^{-\delta(q-1)s}\left\vert y(s)\right\vert ^{q+1}ds. \label{pss}%
\end{equation}
We find
\begin{equation}
\Psi^{\prime}(\tau)=W^{\prime}\left(  \tau\right)  +\frac{\delta(q-1)}%
{q+1}e^{-\delta(q-1)\tau}\left\vert y(\tau)\right\vert ^{q+1}=\mathcal{U}%
(y(\tau),Y(\tau)). \label{pst}%
\end{equation}
Then $\Psi$ is nondecreasing and bounded near $\infty$, thus it has a limit
$\kappa,$ and $W$ has the same limit. And $\mathcal{H}(y,Y)\leq\mathcal{H}%
(\ell+\varepsilon,(\delta(\ell+\varepsilon))^{p-1})=2\delta-N-m,$ for some
$m=m(\varepsilon)>0$, thus
\[
\Psi^{\prime}\left(  \tau\right)  =\mathcal{U}(y\left(  \tau\right)  ,Y\left(
\tau\right)  )\geq m\left(  \delta y-\left\vert Y\right\vert ^{(2-p)/(p-1)}%
Y\right)  \left(  \left\vert \delta y\right\vert )^{p-2}\delta y-Y\right)  .
\]
Now there exists a constant $c=c(p)$ such that for any ($a,b)\in\mathbb{R}%
^{2}\backslash\left\{  (0,0)\right\}  ,$%
\[
\left(  a-b\right)  \left(  \left\vert a\right\vert ^{p-2}a-\left\vert
b\right\vert ^{p-2}b\right)  \geq c(\left\vert a\right\vert +\left\vert
b\right\vert )^{p-2}(a-b)^{2},
\]
thus from (\ref{sys}),
\[
\Psi^{\prime}\left(  \tau\right)  \geq mc\left(  2\delta(\ell+1\right)
)^{p-2}y^{\prime^{2}}(\tau).
\]
Then $y^{\prime^{2}}$ is integrable and bounded; then $\lim_{\tau
\rightarrow\infty}y^{\prime}\left(  \tau\right)  =0.$ Suppose that $y$ admits
an increasing sequence of zeros $\left(  \tau_{n}\right)  $. Then $W(\tau
_{n})=\left\vert Y(\tau_{n})\right\vert ^{p^{\prime}}/p^{\prime}=\left\vert
y^{\prime}(\tau_{n})\right\vert ^{p}/p^{\prime},$ thus $\lim_{\tau
\rightarrow\infty}W\left(  \tau\right)  =0,$ thus $\lim_{\tau\rightarrow
\infty}\mathcal{W}(y(\tau),Y(\tau))=0$. Moreover $\left\vert Y\right\vert
^{(2-p)/(p-1)}Y=\delta y-y^{\prime}=\delta y+o(1),$ thus
\[
\mathcal{W}(y(\tau),Y(\tau))=\frac{(\delta-N)\delta^{p-1}}{p}\left\vert
y(\tau)\right\vert ^{p}-\frac{\delta-\alpha}{2}y^{2}(\tau)+o(1),
\]
which implies $\lim y\left(  \tau\right)  =0$ or $\pm\ell,$ and necssarily
$\lim_{\tau\rightarrow\infty}y\left(  \tau\right)  =0.$ And $\lim
_{\tau\rightarrow\infty}\Psi\left(  \tau\right)  =0,$ thus $\Psi(\tau)\leq0$
near $\infty$, thus
\[
\frac{(\delta-N)\delta^{p-1}}{p}\left\vert y(\tau)\right\vert ^{p}%
-\frac{\delta-\alpha}{2}y^{2}\leq\mathcal{W}(y(\tau),Y(\tau))\leq\frac
{\delta(q-1)}{q+1}%
{\displaystyle\int\limits_{\tau}^{\infty}}
e^{-\delta(q-1)s}\left\vert y(s)\right\vert ^{q+1}ds.
\]
Then $y(\tau)=O(e^{-k_{0}\tau}),$ with $k_{0}=$ $\delta(q-1)/p.$ Assuming that
$y(\tau)=O(e^{-k_{n}\tau}),$ then we find $y(\tau)=O(e^{-k_{n+1}\tau})$ with
$k_{n+1}=k_{n}(q+1)/p+(q-1)/(2-p).$ Since $q>1>p-1,$ it follows that
$y(\tau)=O(e^{-k\tau})$ for any $k>0.$ Consider the substitution
(\ref{cge})for some $d>0.$ Then $y_{d}(\tau)=O(e^{-k\tau})$ for any $k>0.$ At
any maximal point of $\left\vert y_{d}\right\vert $ we find from (\ref{phis})%
\[
(p-1)d(\eta-d)\leq e^{((p-2)d+p)\tau}\left\vert dy_{d}\right\vert
^{2-p}\left(  (\alpha-d)+e^{-d(q-1)\tau}\left\vert y_{d}\right\vert
^{q-1}\right)
\]
Choosing for example $d=\eta/2$ we get a contradiction since the right-hand
sign tends to $0.$\medskip

\noindent(ii) Suppose $N\leq\alpha$ and $\alpha^{\ast}<\alpha.$ Assume that
there exists a solution $w$ with a finite number of zeros. We can assume that
$w(r)>0$ near $\infty.$ From Propositions \ref{alp} and \ref{des}, either
$\lim_{r\rightarrow\infty}r^{\alpha}w=L>0$ or $\lim_{r\rightarrow\infty
}r^{\delta}w=\ell.$ Now the point $M_{\ell}$ is exterior to $\mathcal{S},$
thus $\mathcal{U}(M_{\ell})<0,$ and by computation
\begin{equation}
k_{\ell}:=\mathcal{W}M_{\ell}=\frac{1}{2}\left(  \delta-N\right)  \delta
^{p-2}\ell^{p}=\frac{M}{\left(  \delta-\alpha\right)  ^{\delta}}>0. \label{kl}%
\end{equation}
where $M=M(N,p)=\left(  \delta-N\right)  ^{\delta+1}\delta^{p-2+(p-1)\delta
}/2.\medskip$

$\bullet$ First case: $\lim_{r\rightarrow\infty}r^{\delta}w=\ell.$ Then
$\lim_{\tau\rightarrow\infty}(y(\tau),Y(\tau))=M_{\ell}.$ Thus for large
$\tau,$ $\mathcal{U}(y(\tau),Y(\tau))<0,$ so that $W^{\prime}(\tau)<0.$ Then
$W$ is decreasing, and $\lim_{\tau\rightarrow\infty}W(\tau)=\lim
_{\tau\rightarrow-\infty}\mathcal{W}(y(\tau),Y(\tau))=$ $k_{\ell}.$ Moreover
near $-\infty,$ we find $\lim_{\tau\rightarrow-\infty}W(\tau)=\lim
_{\tau\rightarrow-\infty}\mathcal{W}(y(\tau),Y(\tau))=0;$ indeed near
$-\infty,$ $y(\tau)=O(e^{\delta\tau})$ and $Y(\tau)=O(e^{\delta\tau})$ from
(\ref{pri}) and (\ref{cha}), hence $e^{-\delta(q-1)\tau}\left\vert
y(\tau)\right\vert ^{q+1}=O(e^{2\delta\tau})$. Then $W$ has at least a maximum
point $\tau_{0}$ such that $W(\tau_{0})>k_{\ell}.$ At such a point,
$W^{\prime}(\tau_{0})=0,$ then $\mathcal{U}(y(\tau_{0}),Y(\tau_{0}))>0,$ thus
($y(\tau_{0}),Y(\tau_{0}))\in\mathcal{S}.$ Let $C=\max_{(y,Y)\in
\overline{\mathcal{S}}}(\left\vert y\right\vert +\left\vert Y\right\vert ),$
thus $C=C(N,p)$ and from (\ref{curl}) and (\ref{curi}), and $\max
_{(y,Y)\in\overline{\mathcal{S}}}\mathcal{W}(y,Y)\leq K=K(N,p),$ since
$\alpha-\delta<0.$Then%
\[
k_{\ell}<W(\tau_{0})\leq K+\frac{C^{q+1}}{q+1}%
\]
From (\ref{kl}), it implies that $\delta-\alpha$ is not close to $0.$ More
precisely, there exists $\overline{\alpha}=\overline{\alpha}(N,p)>\max
(N,\alpha^{\ast})$ such that $\alpha\leq\overline{\alpha}.\medskip$

$\bullet$ Second case: $\lim_{r\rightarrow\infty}r^{\alpha}w=L>0.$ It follows
that $\lim_{\tau\rightarrow\infty}e^{\left(  \alpha-\delta\right)  \tau}y=L,$
and $\lim_{\tau\rightarrow\infty}e^{\left(  \alpha-\delta\right)  \tau
}Y=\left(  \alpha L\right)  ^{1/(p-1)},$ from (\ref{imi}). Then $Y(\tau
)=O(y^{p-1}(\tau))$ near $\infty,$ thus%
\[
\mathcal{W}(y(\tau),Y(\tau))+\frac{\delta-\alpha}{2}y^{2}(\tau)=O(y^{p}%
(\tau)),
\]%
\[
W(\tau)+\frac{\delta-\alpha}{2}y^{2}(\tau)=O(y^{p}(\tau))+O(e^{-\delta
(q-1)\tau}y^{q+1}(\tau))=O(y^{p}(\tau))+O(y^{2-\alpha(q-1)/(\delta-\alpha
)}(\tau));
\]
thus $\lim_{\tau\rightarrow\infty}\mathcal{W}(y(\tau),Y(\tau))=\lim
_{\tau\rightarrow\infty}W(\tau)=-\infty;$ and again $\lim_{\tau\rightarrow
-\infty}\mathcal{W}(y(\tau),Y(\tau))=0.$ From \cite[Lemma 4.3]{Bi1} we know
the shape of the level curves $\mathcal{C}_{k}=\left\{  \mathcal{W}%
(y,Y)=k\right\}  :$ either $k>k_{\ell}$ and $\mathcal{C}_{k}$ has two
unbounded connected components, or $0<k<k_{\ell}$ and $\mathcal{C}_{k}$ has
three connected components and one of them is bounded, or $k=k_{\ell}$ and
$\mathcal{C}_{k_{\ell}}$ is connected with a double point at $M_{\ell}$, or
$k=0$ and one of the three connected components of $\mathcal{C}_{0}$ is
$\left\{  \left(  0,0\right)  \right\}  ,$ or $k<0$ and $\mathcal{C}_{k}$ has
two unbounded connected components. As a consequence there exists $\tau_{1}$
such that $\mathcal{W}(y(\tau_{1}),Y(\tau_{1}))=k_{\ell};$ then again
$W(\tau_{1})>k_{\ell}.$ Thus $W$ has at least a maximum point $\tau_{0}$ such
that $W(\tau_{0})>k_{\ell},$ and the conclusion follows as above.
\end{proof}

\section{The case $p\leq(2-p)\alpha$ \label{S4}}

In this section we assume that $p\leq(2-p)\alpha,$ that means $p<2$ and
$\delta\leq\alpha.$

\subsection{Behaviour near infinity}

From Proposition \ref{uti}, we deduce approximate estimates near $\infty$%
\begin{equation}
w(r)=o(r^{-\gamma}),\quad\text{for any }\gamma<\delta. \label{hij}%
\end{equation}
However it is not straightforward to obtain exact estimates, and they can be
false, see Proposition \ref{dis} below. Here again the key point is the use of
enegy function $W$ defined at (\ref{wt}).

\begin{proposition}
\label{tom}Assume $q>1,p<2,$ and $\delta<\alpha,$ or $N\leq\alpha=\delta$.
Then any solution $w$ of problem (\ref{un}) satisfies
\begin{equation}
w(r)=O(r^{-\delta}),\qquad w^{\prime}(r)=O(r^{-\delta-1})\qquad\text{near
}\infty. \label{rd}%
\end{equation}

\end{proposition}

\begin{proof}
\textbf{(i) Case }$\delta<\alpha.$\textbf{ }

$\bullet$ First assume that $2\delta\leq N,$ that means $p\leq p_{2}.$ Then
from (\ref{wtp}), $W^{\prime}(\tau)\leq0$ for any $\tau;$ hence $W$ is bounded
from above near $\infty,$ and in turn $y$ and $Y$ are bounded, because
$\delta<\alpha$ and $p<2.$ Thus (\ref{rd}) holds.$\medskip$

$\bullet$ Then assume $N<2\delta.$ Let $\tau_{0}$ be arbitrary. Since
$\mathcal{S}$ is bounded, there exists $k>0$ large enough such that
$W(\tau)\leq k$ for any $\tau\geq\tau_{0}$ such that $(y(\tau),Y(\tau))\in$
$\mathcal{S},$ and we can choose $k>W(\tau_{0});$ and $W^{\prime}(\tau)\leq0$
for any $\tau\geq\tau_{0}$ such that $(y(\tau),Y(\tau))\not \in $
$\mathcal{S}.$ Then $W(\tau)\leq k$ for any $\tau\geq\tau_{0},$ hence again
$y$ and $Y$ are bounded for $\tau\geq\tau_{0}.\medskip$

\noindent\textbf{(ii) Case }$N\leq\alpha=\delta.$ Since $N<2\delta,$ as above
$W$ is bounded from above for large $\tau.$ We can write $W$ under the form%
\[
W(\tau)=\frac{(\delta-N)\delta^{p-1}}{p}\left\vert y(\tau)\right\vert
^{p}+\Phi(y(\tau),Y(\tau))+\frac{1}{q+1}e^{-\delta(q-1)\tau}\left\vert
y(\tau)\right\vert ^{q+1},
\]
where
\[
\Phi(y,Y)=\frac{\left\vert Y\right\vert ^{p^{\prime}}}{p^{\prime}}-\delta
yY+\frac{\left\vert \delta y\right\vert ^{p}}{p}\geq0,\qquad\forall
(y,Y)\in\mathbb{R}^{2}.
\]
Thus $y$ is bounded, then also $Y$ from H\"{o}lder inequality. $\medskip$
\end{proof}

\begin{remark}
Under the assumptions of Proposition \ref{tom}, we can improve the estimate
(\ref{rd}) for the global solutions: there exists a constant $C=C(N,p)$
\textit{independent on} $a,$ such that all the solutions $w(.a)$ of
(\ref{un}), (\ref{ini}) satisfy
\begin{equation}
\left\vert w(r,a)\right\vert \leq Cr^{-\delta},\qquad\text{for any }r>0.
\label{hol}%
\end{equation}
Indeed let $w$ be any solution. Then $\lim_{\tau\rightarrow-\infty}%
y(\tau)=\lim_{\tau\rightarrow-\infty}Y(\tau)=0,$ thus $\lim_{\tau
\rightarrow-\infty}W(\tau)=0.$ If $2\delta\leq N,$ then $W(\tau)\leq0$ for any
$\tau,$ which gives an upper bound for $y$ independent on $a.$ The same
happens in case $2\delta>N:$ $\mathcal{S}$ is interior to some curve
$\mathcal{W}(y,Y)=k,$ with $k$ independent on $a,$ and $W(\tau)\leq k,$ for
any $\tau$. Thus (\ref{hol}) holds. As a consequence, Then $\left\vert
w(r,a)\right\vert \leq\max(C,a)2^{\delta}(1+r)^{-\delta}$ for any $r>0,$ from
Theorem \ref{exi}.
\end{remark}

The case $\alpha=\delta<N$ is not covered by Proposition \ref{tom}. In fact
(\ref{rd}) is not satisfied, because a logarithm appears:\ 

\begin{proposition}
\label{ego}Assume$q>1,p<2,$ and $\alpha=\delta<N.$ Then any solution $w$ of
(\ref{un})satisfies
\begin{equation}
w=O(r^{-\delta}(\ln r)^{1/(2-p)})\qquad\text{near }\infty. \label{lli}%
\end{equation}

\end{proposition}

\begin{proof}
From (\ref{esu}), we have $w(r)=O(r^{-\delta+\varepsilon})$ for any
$\varepsilon>0,$ hence $y(\tau)=O(e^{\varepsilon\tau});$ and $w$ has a finite
number of zeros, from Proposition \ref{zer},(iv), since $\alpha<N.$ We can
assume that $y$ is positive for large $\tau.$ From (\ref{sys}),
\[
(y-Y)^{\prime}=(N-\delta)Y-e^{\delta(q-1)\tau}y^{q}.
\]
From Lemma \ref{com},(i), $y$ is monotone for large $\tau.$ If $y$ is bounded,
then (\ref{lli}) is trivial. We can assume that $\lim_{\tau\rightarrow\infty
}y=\infty.$ Then also $\lim_{\tau\rightarrow\infty}Y=\infty,$ from Lemma
\ref{com},(iii), and $y^{\prime}\geq0$ for large $\tau,$ hence $Y^{1/(p-1)}%
<\delta y;$ then $Y=o(y)$ near $\infty,$ since $p<2;$ for any $\varepsilon>0,$
$y\leq(1+\varepsilon)(y-Y)$ for large $\tau,$ thus
\[
(y-Y)^{\prime}\leq(N-\delta)(\delta y)^{p-1}\leq(N-\delta)\delta
^{p-1}(1+\varepsilon)^{p-1}(y-Y)^{(p-1)}.
\]
Hence with a new $\varepsilon,$ for large $\tau,$ $(y-Y)^{2-p}(\tau
)\leq(N-\delta)\delta^{p-1}(2-p)(1+\varepsilon)\tau,$ which gives the upper
bound
\begin{equation}
y^{2-p}(\tau)\leq(N-\delta)\delta^{p-1}(2-p)(1+\varepsilon)\tau. \label{thn}%
\end{equation}
In particular (\ref{lli}) holds, and the estimate is more precise:
\begin{equation}
\lim\sup_{r\rightarrow\infty}r^{\delta}(\ln r)^{-1/(2-p)}w\leq((2-p)\delta
^{p-1}(N-\delta))^{1/(2-p)}. \label{prec}%
\end{equation}

\end{proof}

Next we precise the behaviour of the solutions according to the values of
$\alpha$.

\begin{proposition}
\label{dis}. Assume $q>1,p<2.$ Let $w$ be any solution $w$ of problem
(\ref{un}) such that $w$ has a finite number of zeros.

\noindent(i) If\textbf{ }$\delta<\min(\alpha,N),$ then either
\begin{equation}
\lim_{r\rightarrow\infty}r^{\delta}w=\pm\ell, \label{liml}%
\end{equation}

or
\begin{equation}
\lim_{r\rightarrow\infty}r^{\eta}w=c\neq0 \label{limc}%
\end{equation}

or $r^{\delta}w(r)$ is bounded near $\infty$ and $r^{\delta}w$ has no limit,
and
\begin{equation}
\lim_{r\rightarrow\infty}\inf r^{\delta}w\leq\ell\leq\lim\sup_{r\rightarrow
\infty}r^{\delta}w; \label{nol}%
\end{equation}
in the last case $p_{2}<p.$

\noindent(ii) If\textbf{ }$\alpha=\delta<N$, then either
\begin{equation}
\lim_{r\rightarrow\infty}r^{\delta}(\ln r)^{-1/(2-p)}w=\pm\eta,\qquad
\eta=((2-p)\delta^{p-1}(N-\delta))^{1/(2-p)}, \label{loc}%
\end{equation}
or (\ref{limc}) holds.

\noindent(iii) If $\alpha=\delta=N$, then
\begin{equation}
\lim_{r\rightarrow\infty}r^{N}w=k\neq0. \label{lyc}%
\end{equation}

\end{proposition}

\begin{proof}
\textbf{(i) Case }$\delta<\min(\alpha,N).\medskip$

$\bullet$ First assume that $y$ is positive and monotone for large $\tau
$.\ Since it is bounded, from Lemma \ref{com},(ii) and (iv), either
$\lim_{\tau\rightarrow\infty}(y,Y)=M_{\ell}$ and (\ref{liml}) holds; or
$\lim_{\tau\rightarrow\infty}(y,Y)=(0,0),$ thus $y$ is nonincreasing to $0,$
and $\lim_{\tau\rightarrow\infty}y^{\prime}(\tau)=0.$ Comparing to the proof
of Proposition \ref{des}, we observe that (\ref{sit}) is no more true because
$\delta-\alpha<0.$ Nevertheless, for any small $\kappa$ and for $\tau\geq
\tau_{\kappa}$ large enough,
\begin{equation}
-(p-1)y^{\prime\prime}+(\delta p-N)y^{\prime}+(N-\delta-\kappa)\delta y\leq0.
\label{sot}%
\end{equation}
Let us fix $\kappa<N-\delta;$ since $\lim_{\tau\rightarrow\infty}y(\tau)=0,$
we can suppose that $y(\tau)\leq1$ for $\tau\geq\tau_{\kappa}.$ Then there
exists $\mu_{\kappa}<\mu,$ where $\mu$ defined at (\ref{mu}), with
$\mu_{\kappa}=\mu+O(K),$ such that, for any $\varepsilon>0,$ the function
$\tau\longmapsto\varepsilon+e^{-\mu_{\kappa}(\tau-\tau_{\kappa})}$ is a
solution of the corresponding equation on $\left[  \tau_{\kappa}%
,\infty\right)  $.It follows that $y(\tau)\leq\varepsilon+e^{-\mu_{\kappa
}(\tau-\tau_{\kappa})},$ from the maximum principle. Thus $y(\tau)\leq
e^{-\mu_{\kappa}(\tau-\tau_{\kappa})}$ on $\left[  \tau_{\kappa}%
,\infty\right)  $. We can choose $\kappa$ small enough such that $\mu_{\kappa
}(3-p)\geq\mu^{0}:=\mu(4-p)/2>\mu.$ As a consequence, $y(\tau)\leq e^{-\mu
^{0}(\tau-\tau_{\kappa})/(3-p)},$ hence $y^{\prime}(\tau)=O(e^{-\mu^{0}%
\tau/(3-p)}),$ from Proposition \ref{der}. From (\ref{yss}) there exists $C>0$
such that for $\tau\geq\tau_{C}$ large enough, $y(\tau)\leq1$ and
\[
-(p-1)y^{\prime\prime}+(\delta p-N)y^{\prime}+(N-\delta)\delta y\leq
Ce^{-\mu^{0}\tau}.
\]
There exists $A>0$ such that $-Ae^{-\mu^{0}\tau}$ is a particular solution of
the corresponding equation; then $\varepsilon+(1+A)e^{-\mu(\tau-\tau_{C}%
)}-Ae^{-\mu^{0}(\tau-\tau_{C})}$ is also a solution on $\left[  \tau_{\kappa
},\infty\right)  $. Then $y(\tau)\leq\varepsilon+(1+A)e^{-\mu(\tau-\tau_{C})}$
on $\left[  \tau_{\kappa},\infty\right)  $ from the maximum principle, then
$y(\tau)\leq(1+A)e^{-\mu(\tau-\tau_{C})}$. Thus $y(\tau)=0(e^{-\mu\tau}),$
which means $w(r)=O(r^{(p-N)/(p-1)})$ near $\infty.$ As in the proof of
Proposition \ref{des}, $r^{\eta}w$ has a limit $c$ at $\infty,$ and that
$c\neq0.$\medskip

$\bullet$ Next assume that $y$ is positive, but not monotone for large $\tau;$
then there exists an increasing sequence $\left(  \tau_{n}\right)  $ of
extremal points of $y,$ such that $\tau_{n}\rightarrow\infty,$ and (\ref{nol})
follows from Lemma \ref{com}. Assume $p\leq p_{2},$ or equivalently
$2\delta\leq N;$ the function $W$ is nonincreasing hence it has a limit
$\Lambda\geq-\infty.$ Computing at point $\tau_{n},$ where $Y(\tau
_{n})=(\delta y\left(  \tau_{n}\right)  )^{p-1},$ we find
\begin{align*}
W(\tau_{n})  &  =(\alpha-\delta)(\frac{y\left(  \tau_{n}\right)  ^{2}}%
{2}-\frac{\ell^{2-p}y\left(  \tau_{n}\right)  ^{p}}{p})+\frac{1}%
{q+1}e^{-\delta(q-1)\tau_{n}}\left\vert y(\tau_{n})\right\vert ^{q+1}\\
&  =(\alpha-\delta)(\frac{y\left(  \tau_{n}\right)  ^{2}(1+o(1)}{2}-\frac
{\ell^{2-p}y\left(  \tau_{n}\right)  ^{p}}{p}),
\end{align*}
thus $y(\tau_{n})$ has a finite limit, necessarily equal to $\ell$. Then
$\lim_{\tau\rightarrow\infty}y(\tau)=\ell.$\medskip

\noindent\textbf{(ii) Case }$\alpha=\delta<N.$ From Proposition \ref{zer} and
Lemma \ref{com},(i),(ii), $w$ has a finite number of zeros, and $\lim
_{\tau\rightarrow\infty}y=0$ or $\pm\infty$, and (\ref{prec}) holds. If
\textbf{ }$\lim_{\tau\rightarrow\infty}y=\infty,$ we write
\[
(y-Y)^{\prime}+e^{\delta(q-1)\tau}\left\vert y\right\vert ^{q-1}%
y=(N-\delta)Y^{1/(p-1)}Y^{-(2-p)/(p-1)}=(N-\delta)(\delta y-y^{\prime
})Y^{-(2-p)/(p-1)}%
\]
and $Y^{1/(p-1)}<\delta y,$ hence for large $\tau,$
\[
(y-Y)^{\prime}+(N-\delta)Y^{-(2-p)/(p-1)}y^{\prime}\geq y^{p-1}((N-\delta
)\delta^{p-1}-y^{2-p}e^{\delta(q-1)\tau}y^{q-1}..
\]
Since $y^{\prime}\geq0,$ and $\lim_{\tau\rightarrow\infty}Y=\infty,$ for any
$\varepsilon>0$ and for large $\tau,$
\[
(y-Y)^{\prime}+\varepsilon y^{\prime}\geq y^{p-1}((N-\delta)\delta
^{p-1}-e^{\delta(q-1)\tau}y^{q+1-p}).
\]
and $y(\tau)=O(\tau^{1/(2-p)})$ from (\ref{thn}).Thus for any $\varepsilon>0$
and for large $\tau,$%
\[
((1+\varepsilon)y-Y)^{\prime}\geq(N-\delta)\delta^{p-1}(1-\varepsilon
)y^{p-1}.
\]
Setting $\xi=(1+\varepsilon)y-Y,$ we deduce that
\[
\xi^{\prime}\geq(N-\delta)\delta^{p-1}(1-2\varepsilon)\xi^{p-1}%
\]
for large $\tau,$ which leads to the lower bound%
\begin{equation}
y^{2-p}(\tau)\geq(N-\delta)\delta^{p-1}(2-p)(1-3\varepsilon)\tau, \label{low}%
\end{equation}
and (\ref{loc}) follows from (\ref{prec}) and (\ref{low}). If $\lim
_{\tau\rightarrow\infty}y=0,$ (\ref{limc}) follows as in case (i).\medskip

\noindent\textbf{(iii) Case }$\alpha=\delta=N.$ From Proposition \ref{tom},
$y$ and $Y$ are bounded. Moreover $Y-y$ has a finite limit $K,$ and
$Y-y=K+O(e^{-(q-1)\tau}).$ And $y$ has a finite limit limit $l$ from Lemma
\ref{com},(i),(ii). Assume that $l=0.$ Then $\lim_{\tau\rightarrow\infty
}y^{\prime}=-\left\vert K\right\vert ^{(2-p)/(p-1)}K,$ hence $K=0.$ Thus there
exists $C>0$ such that $y^{\prime}=Ny-Y^{1/(p-1)}\geq Ny/2-Ce^{-(q-1)\tau
/(p-1)}$ for large $\tau.$ This implies $y=O(e^{-\gamma_{0}t})$ with
$\gamma_{0}=e^{-(q-1)\tau/(p-1)}.$ Assuming that $y=O(e^{-\gamma_{n}t})$, then
$(Y-y)^{\prime}=O(e^{-(q-1)\tau}y^{q})=O(e^{-(q-1+q\gamma_{n})\tau}),$ hence
$Y=y+O(e^{-(q-1+q\gamma_{n})\tau}).$ Then there exists another $C>0$ such that
$y^{\prime}$ $\geq Ny/2-Ce^{-(q-1+q\gamma_{n})\tau/(p-1)}$ for large $\tau,$
then $y=O(e^{-\gamma_{n+1}t}),$ with $\gamma_{n+1}=(q-1+q\gamma_{n})/(p-1).$
Observe that $\lim\gamma_{n}=\infty,$ thus $y=O(e^{-\gamma t})),$ thus
$w=O(r^{-\gamma}),$ for any $\gamma>0.$ We get a contradiction as in
Proposition (\ref{des}) by using the substitution (\ref{cge}) with $d>N.$
\end{proof}

\subsection{Oscillation or nonoscillation criteria}

As a consequence of Proposition \ref{tom}, we get a first result of existence
of oscillating solutions.

\begin{proposition}
\label{sog} Assume $q>1,p<2,$ and $N\leq\delta<\alpha$ or $N<\delta=\alpha.$
Then for any $m>0,$ any solution $w\not \equiv 0$ of problem (\ref{un}) has a
infinite number of zeros in $\left[  m,\infty\right)  .$
\end{proposition}

\begin{proof}
Suppose that is is not the case. Let $w\not \equiv 0,$ with for example $w>0$
and $w^{\prime}<0$ near $\infty,$ hence $y>0$ and $Y>0$ for large $\tau.$ If
$N<\delta=\alpha$, or $N<\delta=\alpha,$ then $y$ is bounded from Proposition
\ref{tom}. From Lemma \ref{com}, $y$ is monotone, and $\lim_{\tau
\rightarrow\infty}(y(\tau),Y(\tau))=(0,0).$ As in (\ref{moi}), if $N<\delta,$
then $y$ is concave for large $\tau,$ and we reach a contradiction. If
$\delta=N<\alpha,$ we find
\[
(y-Y)^{\prime}=(N-\alpha)y-e^{-\delta(q-1)\tau}\left\vert y\right\vert
^{q-1}y\leq0;
\]
then $y-Y$ is non increasing to $0,$ hence $y\geq Y,$ $Y^{\prime}\geq
NY-Y^{1/(p-1)}\geq NY/2$ for large $\tau,$ which is impossible since
$\lim_{\tau\rightarrow\infty}Y(\tau)=0.\medskip$
\end{proof}

Next we study the case where $\delta<\min(\alpha,N);$ recall that $\delta<N$
$\Leftrightarrow p<p_{1}.$ This case is difficult because the solutions could
be oscillatory, and even if they are not, they have three possible types of
behaviour near $\infty:$ (\ref{liml}), (\ref{limc}), or (\ref{nol}). Here we
extend to equation (\ref{un}) a difficult result obtained in (\cite{Bi1}) for
equation (\ref{hog}). Recall that for system (\ref{aut}), if $\alpha<\eta,$
there exist no solution satisfying (\ref{nol}), and for some $\alpha\in\left(
\eta,\alpha^{\ast}\right)  $ there do exist positive solutions satisfying
(\ref{nol}).

\begin{theorem}
\label{hard}Assume $p_{2}<p<p_{1}$ and $\delta<\alpha$. If $\alpha<\eta,$ (in
particular if $\alpha\leq N),$ then any solution $w(.,a)$ $(a\neq0)$ has a
finite number of zeros and satisfies (\ref{liml}) or (\ref{limc}).
\end{theorem}

\begin{proof}
Assume $\alpha<\eta.$ From Proposition \ref{zer}, (iv), any solution
$w\not \equiv 0$ has a finite number of zeros. We can assume that $w(.,a)$ and
$w^{\prime}(.,a)<0$ for large $r,$ from Proposition \ref{pro}. Consider the
corresponding trajectory $\mathcal{T}_{n}$ of the nonautonomous system
(\ref{sys}) in the phase plane $(y,Y).$ From Proposition (\ref{tom}) it is
bounded near $\infty.$ Let $\Gamma$ be the limit set of $\mathcal{T}_{n}$ at
$\infty;$ then $y\geq0$ and $Y\geq0$ for any $(y,Y)\in\Gamma.$ From \cite{LR},
$\Gamma$ is nonempty, compact and connected, and for any point $P_{0}\in
\Gamma$, the positive trajectory $\mathcal{T}_{a}$ of the autonomous system
(\ref{aut}) issued from $P_{0}$ at time $0$ is contained in $\Gamma.$ From
\cite[Theorem 5.4]{Bi1} we have a complete description of the solutions of
system (\ref{aut}) when $\alpha<\eta.$ Since $\delta<N,$ the point $(0,0)$ is
a saddle point; since $\alpha<\alpha^{\ast}$ the point $M_{\ell}$ is a sink.
The only possible trajectories of (\ref{aut}) ending up in the set $y\geq
0$,$Y\geq0$ are either the points $0,M_{\ell},$ or a trajectory $\mathcal{T}%
_{a,s}$ starting from $\infty$ and ending up at $0$, or trajectories
$\mathcal{T}_{a}$ ending up at $M_{\ell}$. And $\mathcal{T}_{a,s}$ does not
meet the curve
\[
\mathcal{M=}\left\{  (\lambda,(\delta\lambda)^{p-1}):\lambda>0\right\}  .
\]
Then either $\Gamma=\left\{  0\right\}  ,$ or $\Gamma=\left\{  M_{\ell
}\right\}  ,$ or $\Gamma$ contains some point $P_{0}$ of $\mathcal{T}_{a,s},$
or $\mathcal{T}_{a},$ thus also the part of $\mathcal{T}_{a,s}$ or
$\mathcal{T}_{a}$ issued from $P_{0}.$ If $\Gamma=\left\{  M_{\ell}\right\}  $
or $\left\{  0\right\}  ,$ the trajectory converges to this point. If it is
not the case, then $y$ is not monotonous, then there exists a sequence of
extremal points of $y,$ such that $(y,Y)\in\mathcal{M}.$ Let $P_{0}$ be one of
these points$;$ then $P_{0}\not \in \mathcal{T}_{a,s},$ thus the autonomous
trajectory going through $P_{0}$ converges to $M_{\ell}.$ Then $\Gamma$
contains also $M_{\ell},$ thus there exists a sequence $\left(  \tau
_{n}\right)  $ tending to $\infty$ such that $\left(  y\left(  \tau
_{n}\right)  ,Y\left(  \tau_{n}\right)  \right)  $ converges to $M_{\ell}.$
Next we consider again the energy function $W$ defined at (\ref{ww}), and
still use the notations $\mathcal{W},\mathcal{U},\mathcal{H},\mathcal{L}%
,\mathcal{S}$ of Section \ref{S21}. Since $\alpha<\alpha^{\ast},$ the point
$M_{\ell}$ is exterior to the set $\mathcal{S}.$ Thus
\[
\lim W\left(  \tau_{n}\right)  =\mathcal{W}\left(  M_{\ell}\right)  =k_{\ell
}<0,
\]
from (\ref{kl}), since here $\delta<N;$ and $k_{\ell}=\min_{\left(
y,Y\right)  \in\mathbb{R}^{2}}\mathcal{W}\left(  y,Y\right)  ;$ and for large
$n,$ $\left(  y\left(  \tau_{n}\right)  ,Y\left(  \tau_{n}\right)  \right)  $
is exterior to $\mathcal{S},$ thus $\mathcal{U}\left(  y\left(  \tau
_{n}\right)  ,Y\left(  \tau_{n}\right)  \right)  <0,$ thus $W^{\prime}\left(
\tau_{n}\right)  <0.$ Either $W$ is monotone for large $\tau,$ then
$\lim_{\tau\rightarrow\infty}W\left(  \tau\right)  =k_{\ell},$ thus
$\lim_{\tau\rightarrow\infty}\mathcal{W}\left(  \tau\right)  =k_{\ell},$ which
implies $\lim_{\tau\rightarrow\infty}\left(  y\left(  \tau\right)  ,Y\left(
\tau\right)  \right)  =M_{\ell},$ and the trajectory converges to $M_{\ell}.$
Or there exists another sequence $\left(  s_{n}\right)  $ of minimal points of
$W,$ such that $s_{n}>\tau_{n}$ and $W\left(  s_{n}\right)  <W\left(  \tau
_{n}\right)  .$ Then $k_{\ell}\leq\lim\inf\mathcal{W}\left(  s_{n}\right)
\leq\lim\sup\mathcal{W}\left(  s_{n}\right)  =\lim\sup W\left(  s_{n}\right)
\leq k_{\ell}.$ Thus also $\lim_{\tau\rightarrow\infty}\left(  y\left(
s_{n}\right)  ,Y\left(  s_{n}\right)  \right)  =M_{\ell}.$ But
\[
0=W^{\prime}(s_{n})<\mathcal{U}\left(  y\left(  s_{n}\right)  ,Y\left(
s_{n}\right)  \right)
\]
thus $(y\left(  s_{n}\right)  ,Y\left(  s_{n}\right)  )\in\mathcal{S},$ which
is contradictory. Thus $\Gamma=$ $\left\{  M_{\ell}\right\}  $ or $\left\{
0\right\}  ,$ thus $w$ satisfies (\ref{liml}) or (\ref{limc}) from Proposition
(\ref{dis}).\medskip
\end{proof}

\begin{remark}
If $\alpha>\alpha^{\ast},$ the regular solutions of system (\ref{aut}) are
oscillatory, see \cite[Theorem 5.8]{Bi1}. We cannot prove the same result for
equation (\ref{un}), since it is a global problem, and system (\ref{sys}) is
only a perturbation of (\ref{aut}) near infinity; and the use of the energy
function $W$ does not allow to conclude.
\end{remark}

\subsection{Existence of positive solutions}

From Theorem \ref{hard}, we first prove the existence of positive solutions,
and their decay can be qualifieed as slow among the possible behaviours given
at Proposition \ref{dis}:

\begin{proposition}
\label{oro}Assume $\delta\leq\alpha<N$. Let $\underline{a}>0$ be defined at
Proposition \ref{zer}. Then for any $a\in\left(  0,\underline{a}\right]  $,
and $w(r,a)>0$ on $\left[  0,\infty\right)  ,$ and satisfies (\ref{liml}) if
$\delta<\alpha$, or (\ref{loc}) if $\alpha=\delta$.
\end{proposition}

\begin{proof}
We still have $w(r,a)>0$ from Proposition \ref{zer}, and $J_{N}$ is
nondecreasing and $J_{N}(0)=0.$ If the conclusions were not true, then
$w(r)=O(r^{-\eta}),$ from Theorem \ref{hard}, then $r^{N}w=O(r^{N-\eta})$, and
$N<\eta$ from (\ref{dn}). Then $\lim\sup_{r\rightarrow\infty}J_{N}(r)$
$\leq0,$ and we reach a contradiction as at Proposition \ref{ora}.
\end{proof}

Next we show the existence of positive solutions with a (faster) decay in
$r^{-\eta}$ in the subcritical case:

\begin{theorem}
\label{fast2}Assume $p<2,$ $\delta<\alpha<N,$ and $1<q<q^{\ast}.$ Then there
exists $a>0$ such that $w(.,a)$ is positive and satisfies $\lim_{r\rightarrow
\infty}r^{\eta}w=c\neq0$.
\end{theorem}

\begin{proof}
Let
\[
A=\left\{  a>0:w(.,a)>0\;\text{on }\left(  0,\infty\right)  \;\text{and }%
\lim_{r\rightarrow\infty}r^{\delta}w=\ell\right\}  ,\text{ }%
\]%
\[
B=\left\{  a>0:w(.,a)\quad\text{has at least an isolated zero}\right\}  .
\]
Then $A$ and $B$ are nonempty Propositions \ref{oro} and \ref{sig}, and
$A\supset\left(  0,\underline{a}\right]  $ and $B\supset\left[  \overline
{a},\infty\right)  ,$ and $B$ is open. Now we show that $A$ is open. Let
$a_{0}\in A.$ Then $J_{N}(.,a_{0})$ is increasing for large $r$ and tends to
$\infty,$ thus $J_{N}(r_{0},a_{0})>0$ and $J_{N}^{\prime}(r_{0},a_{0})>0$ for
$r_{0}$ large enough; and then there exists a neighborhood $\mathcal{V}$ of
$a_{0}$ such that $w(r,a)>0$ on $\left[  0,r_{0}\right]  $ and $J_{N}%
(r_{0},a)>0$ and $J_{N}^{\prime}(r_{0},a)>0$ for any $a\in\mathcal{V}.$ Then
$J_{N}^{\prime}(r_{0},a)>0$ for any $r\geq r_{0},$ since $w(.,a)$ is
decreasing. Then for any $a\in\mathcal{V}$, from Propositions \ref{dis} and
\ref{der}, either $\lim_{r\rightarrow\infty}r^{\eta}w=c>0,$ and $\lim
_{r\rightarrow\infty}r^{\eta+1}w^{\prime}=-c\eta,$ from (\ref{sysd}) and
(\ref{cgc}) with $d=\eta;$ then $\lim_{r\rightarrow\infty}J_{N}(.,a)=-c^{p-1}%
,$ which is impossible. Or necessarily $\lim_{r\rightarrow\infty}r^{\delta
}w(.,a)=\ell,$ thus $a\in A$. Let $a_{\inf}=\inf B>\underline{a}$ and
$a_{\sup}=\sup A<\overline{a}.$ Taking $a=a_{\inf}$ or $a_{\sup},$ then
$w(.,a)$ is positive and $\lim_{r\rightarrow\infty}r^{\eta}w=c.\medskip$
\end{proof}

\begin{remark}
\label{non}Under the assumptions of theorem \ref{fast2}, any solution $w(.,a)$
$(a\neq0)$ has a finite number of zeros, and $\lim_{r\rightarrow\infty
}r^{\delta}w(.,a)=\Lambda(a),$ with $\Lambda(a)=\pm\ell$ or $0.$ Here the
function $\Lambda$ is not continuous on $\left(  0,\infty\right)  .$ Indeed it
would imply that the set $\left\{  a>0:\Lambda(a)=\ell\right\}  $ is closed
and open in $\left(  0,\infty\right)  ,$ and non empty, which contradicts the
above results.\medskip
\end{remark}

At last in the supercritical case, we show the existence of grounds states for
any $a>0,$ and they have a (slow) decay:

\begin{theorem}
\label{pol}Assume $\delta\leq\alpha$. Let $w(r,a)$ be the solution of problem
(\ref{un}), (\ref{ini}).

\noindent(i) If $p\leq p_{2},$ then for any $a>0,$ $w(r,a)>0$ on $\left[
0,\infty\right)  $ and (\ref{liml}) or (\ref{loc}) holds.

\noindent(ii) If $p_{2}<p<p_{1}$ and $\alpha<(N-1)p^{\prime}/2$, and $q\geq
q_{\alpha}^{\ast}>q^{\ast},$ where $q_{\alpha}^{\ast}$ is given by
(\ref{eto}), then again $w(r,a)>0$ on $\left[  0,\infty\right)  $ and
(\ref{liml}) or (\ref{loc}) holds.
\end{theorem}

\begin{proof}
We consider again the function $V=V_{\lambda,\sigma,e}$ defined at
(\ref{vla}).$\medskip$

\noindent(i) Suppose $p\leq p_{2}.$ As in Theorem \ref{pos} (ii) we take
$\lambda=N=2\sigma$ and $e=\alpha-N/2.$ Then $V^{\prime}\leq0$ from
(\ref{wac}) and in the same way $w(r)>0$ on $\left[  0,\infty\right)  .$ From
Proposition (\ref{dis}), if (\ref{liml}) does not hold, then $w=O(r^{-\eta}),$
$w^{\prime}=O(r^{-(\eta+1)})$ near $\infty$. Then by computation,
$V(r)=O(r^{-\eta}),$ thus $\lim_{r\rightarrow\infty}V(r)=0.$ Then $V\equiv0$
on $\left[  0,\infty\right)  $ which is contradictory.\medskip

\noindent(ii) Suppose $p_{2}<p<p_{1},$ and $\alpha<(N-1)p^{\prime}/2.$ As in
Theorem \ref{pos} (ii) we take $\lambda=2\alpha$ and $\sigma=N-1-2\alpha
/p^{\prime}$ and $e=\sigma-\alpha.$Observe that $\alpha<\eta,$ thus from
Theorem \ref{hard}, if (\ref{liml}) does not hold, then again $w=O(r^{-\eta
}),$ $w^{\prime}=O(r^{-(\eta+1)})$ near $\infty$. Then by computation,
$V(r)=O(r^{2\alpha-(N-1)p^{\prime}})$ near $\infty,$ hence $\lim
_{r\rightarrow\infty}V(r)=0$ and we reach again a contradiction.
\end{proof}

\section{Back to problem (\ref{lap})\label{S5}}

Here we apply to equation (\ref{pre}) the results of Section 3 with
$\alpha=\alpha_{0}=p/(q+1-p),$ and show our main result.\medskip

\begin{proof}
[Proof of Theorem \ref{prin}]0ne has $\alpha_{0}>0$ since $q>p-1,$ and
(\ref{log}) holds since $q>1.$

\noindent(i) The existence and behaviour of $w$ follows from Theorem \ref{exi}
and Proposition \ref{alp}.\medskip

\noindent(ii) Condition $q_{1}<q$ is equivalent to $\alpha_{0}<N,$ and
Proposition \ref{ora} applies.\medskip

\noindent(iii) If $q_{1}<q<q^{\star},$ then Theorem \ref{fast} shows the
existence of fast nonnegative decaying solutions $w$. For any $s\geq1,$ there
exists $C>0$ such that for any $t>0,$%
\begin{equation}
\left\Vert u(t)\right\Vert _{s}=Ct^{(N/s\alpha_{0}-1)/(q-1)}\left\Vert
w\right\Vert _{s}. \label{tns}%
\end{equation}
If $p>2,$ then $w$ has a compact support thus $u(t)\in L^{s}(\mathbb{R}^{N}).$
If $p<2,$ then $u$ is positive, and from Proposition (\ref{des}), $w$
satisfies \ref{kil}, with $\ell(N,p,q)$ and $\rho(N,p,q)$ given by (\ref{lim})
and (\ref{lam}) with $\alpha=\alpha_{0}:$%
\[
\ell(N,p,q)=\left(  \delta^{p-1}\frac{\delta-N}{\delta-\alpha_{0}}\right)
^{1/(2-p)}\qquad\rho(N,p,q)=\frac{1}{N}\left(  \frac{N(N-1)}{2(N-\alpha_{0}%
)}\right)  ^{(N+1)/2};
\]
hence again $u(t)\in L^{s}(\mathbb{R}^{N})$. Indeed either $p_{1}<p,$ thus
$N<\delta,$ and $w=O(r^{-\delta})$ at $\infty,$ thus $%
{\displaystyle\int\nolimits_{1}^{\infty}}
r^{N-1-\delta s}dr<\infty;$ or $p<p_{1},$ thus $w=O(r^{-\eta})$ and $N<\eta,$
thus $%
{\displaystyle\int\nolimits_{1}^{\infty}}
r^{N-1-(N-p)s/(p-1)}dr<\infty;$ or $p=p_{1},$ and $w=O(r^{-N}(\ln
r)^{-(N+1)/2}),$ and $%
{\displaystyle\int\nolimits_{1}^{\infty}}
r^{N-1-Ns}(\ln r)^{-(N+1)/2}dr<\infty.$ Moreover $\lim_{t\rightarrow
0}\left\Vert u(t)\right\Vert _{s}=0$ whenever $s>N/\alpha_{0},$ from
(\ref{tns}). For fixed $\varepsilon>0,$ from Proposition \ref{tim}, either
$p>2$ and $\sup_{\left\vert x\right\vert \geq\varepsilon}\left\vert
u(x,t)\right\vert =0$ for $t\leq t(\varepsilon)$ small enough, or $p<2$ and
$\sup_{\left\vert x\right\vert \geq\varepsilon}\left\vert u(x,t)\right\vert
\leq C(\varepsilon)t^{(\delta/\alpha_{0}-1)/(q-1)}$ for $t\leq t(\varepsilon)$
small enough, and $\alpha_{0}<\delta,$ hence in any case $\lim_{t\rightarrow
0}\sup_{\left\vert x\right\vert \geq\varepsilon}\left\vert u(x,t)\right\vert
=0.$\medskip

\noindent(iv) The assertions follow from Theorem \ref{pn} (ii) and (iii), and
from Proposition \ref{des}.\medskip

\noindent(v) Here we applyTheorem \ref{pos} (i) and (ii). Indeed if $p>p_{2},$
and $q\geq q^{\star},$ then $\alpha_{0}\leq(N-p)/p<N/2$.\medskip

\noindent(vi) If $1<q\leq q_{1},$ then $N<\delta$ and $N\leq\alpha_{0}$ thus
all the solutions $w$ are changing sign, from Proposition \ref{zer}, (ii); and
there exists an infinity of fast decaying solutions $w$, from Theorem \ref{pn}
(ii); the estimates follow from Proposition \ref{tim}. Moreover in the case
$p<2,$ from Theorem \ref{osci}, $w$ has a finite number of zeros if
$\alpha_{0}$ is not too large, in particular if $\alpha_{0}<\alpha^{\ast}%
,$where $\alpha^{\ast}$ is defined at (\ref{eto}) ($\alpha^{\ast}<\delta),$
which means $1<p-1+p/\alpha^{\ast}<q\leq q_{1}.$This requires $N<\alpha^{\ast
},$ which means that $p$ is sufficiently close from $2$ , more precisely
$(2p-3)p>N(2-p)(p-1),$ in particular $p>3/2).$ On the contrary, there exists
$\bar{\alpha}\in(\max(N,\alpha^{\ast}),\delta)$ such that $w$ is oscillatory
if $\alpha_{0}>\bar{\alpha},$ which means $1<q<p-1+p/\bar{\alpha}.$
\end{proof}

\begin{remark}
If $q=q_{1},$ then $\alpha_{0}=N,$ thus for each of these functions $w,$ there
exists $C\in\mathbb{R}$ such that the corresponding function $u$ satisfies $%
{\displaystyle\int\nolimits_{\mathbb{R}^{N}}}
u(t)dx=C%
{\displaystyle\int\nolimits_{\mathbb{R}^{N}}}
wdx,$ and $\left\Vert u(t)\right\Vert _{1}=\left\vert C\right\vert \left\Vert
w\right\Vert _{1}$ for any $t>0;$ then there exists a sequence $\left(
t_{n}\right)  \rightarrow0$ such that $u(t_{n})$ converges weakly to a bounded
measure $\mu$ in $\mathbb{R}^{N};$ we still have $\lim_{t\rightarrow0}%
\sup_{\left\vert x\right\vert \geq\varepsilon}\left\vert u(x,t)\right\vert
=0,$ hence $\mu$ has its support at the origin; we cannot assert that $\mu$ is
a Dirac mass as in the case $p=2,$ see \cite{W1}, since we have no uniqueness
result for equation \ref{lap}, inasmuch as $u$ has not a constant sign.
\end{remark}

\end{document}